\documentclass[11pt,reqno]{amsart}
\usepackage{amsbsy,amsfonts,amsmath,amssymb,amscd,amsthm}
\usepackage{mathrsfs,float,color}
\usepackage[mathcal]{euscript}
\usepackage{subfigure,enumitem,graphicx,epstopdf}
\usepackage[a4paper,text={6.1in,8in},centering,top=1.5in]{geometry}
\usepackage{pxfonts,epstopdf}
\usepackage[colorlinks=true,linkcolor=blue,citecolor=green]{hyperref}
\usepackage{srcltx}
\usepackage[margin=20pt,font=small,labelfont=bf,textfont=it]{caption}
\usepackage{overpic}
\usepackage{bbm}

\small\normalsize
\theoremstyle{plain}
\newtheorem{mainthm}{Theorem}

\newtheorem{thm}{Theorem}[section]
\newtheorem{cor}[thm]{Corollary}
\newtheorem{lem}[thm]{Lemma}
\newtheorem{claim}[thm]{Claim}
\newtheorem{prop}[thm]{Proposition}

\theoremstyle{definition}

\newtheorem{rem}[thm]{Remark}

\newcommand{\eqdef}{\stackrel{\scriptscriptstyle\rm def}{=}}

\newcommand{\N}{\mathbb{N}}

\numberwithin{equation}{section}

\begin{document}
\title[Periodic points and measures for a class of skew products]{Periodic points and measures for a class of skew products}
\author[M. Carvalho]{Maria Carvalho}
\address{Maria Carvalho\\ Centro de Matem\'{a}tica da Universidade do Porto\\ Rua do
Campo Alegre 687\\ 4169-007 Porto\\ Portugal}
\email{mpcarval@fc.up.pt}

\author[S. A. P\'{e}rez]{Sebasti\'{a}n A. P\'{e}rez}
\address{Sebasti\'{a}n A. P\'{e}rez Opazo \\ Centro de Matem\'{a}tica da Universidade do Porto\\ Rua do
Campo Alegre 687\\ 4169-007 Porto\\ Portugal}
\email{sebastian.opazo@fc.up.pt}

\date{\today}
\thanks{This research was partially supported by CMUP (UID/MAT/00144/2019) which is funded by FCT with national (MCTES) and European structural funds through the programs FEDER, under the partnership agreement PT2020. SP also acknowledges financial support from the postdoctoral grant PTDC/CTM/BIO-4043-2014, under the project UID/MAT/00144/2019.}

\begin{abstract}
We consider the open set constructed by M. Shub in \cite{S1971} of partially hyperbolic
skew products on the space $\mathbb{T}^2 \times \mathbb{T}^2$ whose non-wandering set is not stable.
We show that there exists an open set $\mathcal{U}$ of such diffeomorphisms such that if $F_S \in \mathcal{U}$ then its measure of maximal entropy is unique, hyperbolic and, generically, describes the distribution of periodic points. Moreover, the non-wandering set of such an $F_S \in \mathcal{U}$ contains closed invariant subsets carrying entropy arbitrarily close to the topological entropy of $F_S$ and within which the dynamics is conjugate to a subshift of finite type. Under an additional assumption on the base dynamics, we verify that $F_S$ preserves a unique SRB measure, which is physical, whose basin has full Lebesgue measure and coincides with the measure of maximal entropy. We also prove that there exists a residual subset $\mathfrak{R}$ of $\mathcal{U}$ such that if $F_S \in \mathfrak{R}$ then the topological and periodic entropies of $F_S$ are equal, $F_S$ is asymptotic per-expansive, has a sub-exponential growth rate of the periodic orbits and admits a principal strongly faithful symbolic extension with embedding.
\end{abstract}

\maketitle

\setcounter{tocdepth}{2}

\section{Introduction}
Let $f: M \to M$ be a diffeomorphism of a manifold into itself and $\Omega(f)$ be its non-wandering set. When $\Omega(f)$ does not admit a hyperbolic structure, it may be difficult to describe completely its orbit structure. Motivated by this problem, R. Bowen suggested to look for invariant components of $\Omega(f)$ with large entropy on which the dynamics of $f$ may be simpler to characterize. The key idea is to find closed invariant subsets, say \textit{topological horseshoes}, within which the dynamics is conjugate to subshifts of finite type that may be good approximations, in some sense, of the global dynamics. For instance, this strategy might provide information on the topological entropy of a complicated dynamics by taking the least upper bound over its restrictions to those horseshoes. In this case, the system is said to be a \textit{limit of horseshoes in the sense of the entropy}. L.-S. Young studies in \cite{Y1981} systems that are limits of this type, including piecewise monotonic maps of the interval, the Poincar\'e map of the Lorenz attractor \cite{Guc1976} and Abraham-Smale's examples \cite{AS1970}, leaving unsolved the case of the partially hyperbolic, robustly transitive, entropy-expansive and non-$\Omega$-stable skew products introduced by Shub in \cite{S1971}. In this work we consider precisely these skew products, explore the dynamical properties of their measures of maximal entropy and thereby show, on Section~\ref{se.properties}, that they are indeed limits of horseshoes.

The second question we address here concerns the study of the distribution of periodic points and measures of maximal entropy. Denote by $\mathcal{U}$ the open set of Shub's examples as constructed in \cite{NY1983} and by $F_S$ any of its elements. It is known that $F_S$ has a unique measure of maximal entropy, and so one expects this measure to have a strong tie with other dynamical properties.
In particular, it would be relevant to show that this measure describes the distribution of the periodic points of $F_S$ (meaning that it is the weak$^*$-limit of the sequence of Dirac measures supported on the sets of periodic points) and that the asymptotic exponential growth rate of the number of periodic orbits with the period (the so called \emph{periodic entropy}) is equal to the topological entropy of $F_S$. We will prove that these two attributes, which are known to be valid within the uniformly hyperbolic setting (cf. \cite{B1970}), also hold on a residual subset $\mathfrak{R}$ of $\mathcal{U}$. Both properties are a consequence of the existence of a semi-conjugation between $F_S$ and a uniformly hyperbolic dynamics, besides a careful analysis of the periodic fibers induced by the semi-conjugation. Thereby, our study conveys a satisfactory description of the symbolic dynamics of Shub's examples. More precisely, we show that in $\mathfrak{R}$ any diffeomorphism has a sub-exponential growth rate of the periodic orbits in arbitrarily small scales (the so called \emph{asymptotic per-expansiveness}); this result enables us to build a
symbolic extension, from whose properties we conclude that, generically in $\mathcal{U}$, the set of Borel invariant probability measures is homeomorphic to the space of Borel probability measures invariant by a subshift of finite type.
%
The proofs of these assertions will be presented on Sections~\ref{se.proof-Theorem-A} and \ref{se.proof-Theorem-B}.

For Anosov diffeomorphisms and, more generally, $C^2$ Axiom A attractors, the work of Bowen, Ruelle and Sinai (we refer the reader to \cite{B1975} and references therein) proved the existence of a unique invariant probability measure, the so-called SRB measure, that is characterized by obeying Pesin's formula \cite{Pes1977}. From Ledrappier and L.-S. Young's work \cite{LY1985}, the property that defines an SRB measure is known to be equivalent to the existence of a disintegration of the measure in conditional measures on unstable manifolds which are absolutely continuous with respect to Lebesgue measure. Moreover, the SRB measure is also the unique physical measure (cf. \cite[Theorem 4.12]{B1975}; a thorough essay on the existence and uniqueness of both SRB and physical measures within more general settings may be read in \cite{Y2002}). For Shub's examples, the existence of an SRB measure was proved in \cite{CY2005}. Besides, under the additional assumption that the base dynamics is the product of two linear hyperbolic automorphisms of the $2$-torus, $F_S$ is mostly contracting with a minimal strong unstable foliation, and so it has a unique SRB measure whose basin of attraction has full Lebesgue measure (cf. \cite{BV2000}). Consequently, the SRB measure of $F_S$ is also its unique physical measure and coincides with its measure of maximal entropy, thus inheriting this property from the conservative base dynamics. 
More detailed information will be given on Section~\ref{se.proof-Theorem-C}.


\section{Main results}\label{results}

It is known that a diffeomorphism $f:M\to M$ on a compact Riemannian manifold $M$ satisfying the Axiom A condition and without cycles is expansive and has the specification property. So it preserves a unique invariant probability measure with maximal entropy which describes the distribution of the periodic points. Moreover, for these systems the topological and periodic entropies are equal (cf. \cite{B1971}). Summarizing:
\begin{itemize}
\item \emph{Uniqueness}: $f$ preserves a unique probability measure $\mu$ satisfying $h_{\mu}(f)=h_{\mathrm{top}}(f)$, where $h_{\mathrm{top}}(f)$ denotes the topological entropy of $f$ (definition in \cite{W1981}).
\\
\item \emph{Distribution of periodic points}: $\mu$ is the limit in the weak$^*$ topology of the sequence of equidistributed averages supported on the periodic points of $f$.
\\
\item \emph{Equal topological and periodic entropies}: $\lim_{n\, \to \, +\infty}\,\frac{1}{n}\,\log\,\#\, \mathrm{Per}_n(f) = h_{\mathrm{top}}(f)$, where $\mathrm{Per}_n(f)$ stands for the set of $n$th periodic points of $f$.
\\
\item \emph{Symbolic dynamics}: There exists a principal strongly faithful symbolic extension with embedding (reminding a similar property valid for Axiom A systems -- see~\cite{Man,B2017}).
\end{itemize}

Since specification and expansivity are not valid in general outside the hyperbolic world, the previous properties are not expected within this setting (cf. \cite{BDF,Ka1999}). Nonetheless, in a broad class of non-hyperbolic systems the existence of at least one probability measure of maximal entropy is also guaranteed. For instance, this is valid for entropy-expansive diffeomorphisms (cf. \cite{Mz1976}). And it was shown in \cite{CY2005} (see also \cite{DF2011, DFPV2012} for generalizations) that, when the central bundle of $f$ is one-dimensional, then $f$ is entropy-expansive. So Shub's examples are endowed with a probability measure of maximal entropy. However, even if we assume that the system is topologically mixing, uniqueness of such a special measure is not certain (cf. \cite{HHTU2012}). For Shub's examples the uniqueness of the measure of maximal entropy was obtained in \cite{NY1983} (a generalization for equilibrium states may be read in \cite{CP2018}). Nevertheless, without additional assumptions this measure may not describe the distribution of the periodic points and the topological entropy may be different from the periodic one. Yet, as we will explain, Shub's examples, which are obtained through a homotopic deformation of a direct product of two hyperbolic diffeomorphisms,
may be constructed in such a way that, if we restrict to a suitable residual subset of them, then we are able to keep control on the periodic orbits even at arbitrarily small scales.
This is our first result. 

Denote by $\mathcal{U}$ the open set of Shub's examples, whose construction will be recalled on Section~\ref{se.Shub-examples}, and by $F_S$ any of its elements.

\begin{mainthm}\label{teo.A} There exists a residual subset $\mathfrak{R}$ of the open set $\mathcal{U}$ such that, if $F_S$ belongs to $\mathfrak{R}$, then:
\begin{enumerate}
\item[(a)] $h_{\mathrm{top}}(F_S) = \lim_{n \,\to \, +\infty}\,\frac{1}{n}\,\log\, \#\,\mathrm{Per}_n(F_S).$
\medskip
\item[(b)] The maximal entropy measure of $F_S$ describes the distribution of periodic points.
\medskip
\end{enumerate}
\end{mainthm}

As previously mentioned, Shub's examples are entropy-expansive, and this property is a sufficient condition for the existence of a principal symbolic extension. However, if we restrict to $\mathfrak{R}$, the diffeomorphisms satisfy a stronger property, namely the asymptotic per-expansiveness, and such an extension may be constructed preserving the periodic points and inducing a homeomorphism between the corresponding spaces of probability measures.

\begin{mainthm}\label{teo.B} Any diffeomorphism of the residual subset $\mathfrak{R}$ built in Theorem~\ref{teo.A} has a principal strongly faithful symbolic extension with embedding.
\end{mainthm}



The construction of Shub's examples ensures that, if both $\Phi$ and $L$ are linear hyperbolic automorphisms of $\mathbb{T}^2$, then the map $F_S$ is mostly contracting and so, according to \cite{BV2000}, it has a finite number of SRB measures whose basins cover Lebesgue almost everywhere. In addition, the strong unstable foliation of $F_S$ is robustly minimal (cf. Proposition~\ref{prop.minimal}), so $F_S$ has in fact a unique SRB measure and its basin has full Lebesgue measure. Under this additional assumption on $\Phi$, the map $F_S$ inherits from $\Phi \times L$ other properties.

\begin{mainthm}\label{teo.C} Suppose that both $\Phi$ and $L$ are linear hyperbolic automorphisms of $\mathbb{T}^2$. Then:
\begin{enumerate}
\item[(a)] The image by $H_*$ of the SRB measure of $F_S\in \mathcal{U}$ is the SRB measure of $\Phi \times L$.
\medskip
\item[(b)] The SRB measure of $F_S \in \mathcal{U}$ is its unique measure of maximal entropy and its unique physical measure.
\end{enumerate}
\end{mainthm}

\section{Glossary}\label{se.definitions}

We begin introducing the main definitions used in this work. Given a compact metric space $(X,d)$ and a continuous map $f: X \to X$, denote by $\mathscr{P}(X)$ the set of Borel probability measures on $X$ endowed with the weak$^*$-topology, and by $\mathscr{P}(X,f)$ and $\mathscr{P}_e(X,f)$ its subsets of $f$-invariant and $f$-invariant ergodic elements, respectively.

\subsection{Maximal entropy measures}
For each $\mu$ in $\mathscr{P}(X,f)$, consider the metric entropy $h_\mu(f)$ of $f$ with respect to $\mu$. The Variational Principle \cite[Theorem 9.10]{W1981} states that the topological entropy  $h_{\mathrm{top}}(f)$ of $f$ coincides with the least upper bound of the operator $\mu \,\mapsto \,h_\mu(f)$ restricted to either $\mathscr{P}(X,f)$ or $\mathscr{P}_e(X,f)$. A measure $\mu \in \mathscr{P}(X,f)$ such that $h_{\mu}(f) = h_{\mathrm{top}}(f)$ is called a \textit{measure of maximal entropy} of $f$.

\subsection{Distribution of periodic points}
Assume that the cardinality $\#\,\mathrm{Per}_n(f)$ of the set of fixed points of $f^n$ is finite for every $n\in \mathbb{N}$. We say that a probability measure $\mu\in \mathscr{P}(X,f)$ \textit{describes the distribution of the periodic points of} $f$ if $\mu$ is the weak$^*$ limit of the sequence of probability measures
$$n \in \mathbb{N} \quad \mapsto \quad \frac{1}{\#\,\mathrm{Per}_n(f)} \sum_{x \,\in\, \mathrm{Per}_n(f) }\delta_x$$
where $\delta_x$ denotes the Dirac measure supported at $x$.


\subsection{Expansiveness}\label{ss.ent-exp}

Denote by $B_\rho(x)$ the open ball in the metric $d$ centered at $x$ with radius $\rho$, and by $\overline{B_\rho(x)}$ its closure. Define, for each $n \in \mathbb{N}$, the equivalent metric
$$(x,y) \in X \times X \quad \mapsto \quad d_n(x,y) \,\eqdef\,\, \max_{0 \,\leqslant \, j \,\leqslant \,n-1} \,d(f^j(x), \,f^j(y)).$$

Given $\varepsilon> 0$ and a compact subset $Y \subset X$, a subset $S$ of $Y$ is said to be $(n,\varepsilon)$-\textit{spanning} if for every $y \in Y$ there is $a \in S$ such that $d_n(y,a) \leqslant \varepsilon$. The minimum cardinality of the $(n, \varepsilon)$-spanning subsets of $Y$ is denoted by $r_n(Y, \varepsilon)$. Define
\begin{eqnarray*}
\overline{r}_n(Y,\varepsilon) \,\eqdef\,\,  \limsup_{n \, \to \, +\infty}\, \frac{1}{n}\,\log \, r_n(Y, \varepsilon) \quad \quad \mbox{and} \quad \quad
\overline{h}_{\mathrm{top}}(f,Y) \,\eqdef\, \lim_{\varepsilon \, \to \, 0^+}\, \overline{r}_n(Y,\varepsilon).
\end{eqnarray*}
Having fixed $\varepsilon> 0$ and $x \in X$, consider the set of points in $X$ whose forward orbits by $f$ are $\varepsilon$-close to the orbit of $x$, that is,
$$B^f_{\infty,\,\varepsilon}(x) \,\eqdef\,\, \,\bigcap_{i \,\in \,\mathbb{A}}\,
f^{-i}\,\Big(\,\overline{B_{\varepsilon}(f^{i}(x))}\Big)=\Big\{y\in X: \,\, d(f^i(x),f^i(y))\leqslant \varepsilon,\quad \forall \, i \in \mathbb{A}\Big\}
$$
with $\mathbb{A}=\mathbb{Z}$ if $f$ is invertible and $\mathbb{A}=\mathbb{N}$ otherwise. Consider
$$h^*_{\mathrm{top}}(f, \varepsilon) \,\eqdef\,\, \,\sup_{x \,\in\,X}\, \overline{h}_{\mathrm{top}}(f,B^f_{\infty,\,\varepsilon}(x))\qquad \mbox{and} \qquad h^*_{\mathrm{top}}(f)\eqdef \lim_{\varepsilon\,\to\, 0^+}\,h^*_{\mathrm{top}}(f, \varepsilon).
$$
The map $f$ is said to be \emph{entropy-expansive} if there is $\varepsilon_0 > 0$ such that $h^*_{\mathrm{top}}(f, \varepsilon)= 0$ for every $0 < \varepsilon < \varepsilon_0$, and \textit{asymptotically entropy-expansive} if $h^*_{\mathrm{top}}(f)=0$. Misiurewicz has shown in \cite{Mz1976} that for asymptotically entropy-expansive maps the entropy operator $\mu \in \mathscr{P}(X, f)\to h_{\mu}(f)$ is upper-semicontinuous, guaranteeing the existence of at least a measure of maximal entropy for $f$.

Given $\varepsilon > 0$, define
$$Per(f, \varepsilon)\,\eqdef\,\, \limsup_{n \, \to \, +\infty}\,\frac{1}{n}\,\sup_{ x \, \in \, X}\,
\log \,\#\,\Big(\mathrm{Per}_n(f) \cap B^f_{\infty,\,\varepsilon}(x)\Big) \qquad \mbox{and} \qquad Per^*(f)\eqdef \lim_{\varepsilon\,\to\, 0^+}\, Per(f, \varepsilon).$$
Following \cite{B2017}, the map $f$ is said to be \textit{asymptotically per-expansive} if $Per^*(f) = 0$. For instance, expansive or aperiodic maps are asymptotically per-expansive.
An interesting connection between the entropy, the growth of the cardinality of the periodic orbits with the period and the asymptotic per-expansiveness is given in the next lemma.

\begin{lem}\cite[Lemma 2.2]{Burguet2017}
$\quad \limsup_{n \, \to \, +\infty}\,\frac{1}{n}\, \log \,\#\,\mathrm{Per}_n(f) \leqslant h_{\mathrm{top}}(f) + Per^*(f).$
\end{lem}

Thus, if $f$ is asymptotically per-expansive then $\limsup_{n \, \to \, +\infty}\,\frac{1}{n}\,
\log \,\#\,\mathrm{Per}_n(f)\leqslant h_{\mathrm{top}}(f)$, an inequality that generalizes \cite[Theorem 8.16]{W1981}.


\subsection{Symbolic extensions}\label{ss.symbolic-ext}
A map $f$ has a \emph{symbolic extension} if there exists $m \in \mathbb{N}$, a closed shift-invariant subset $\Sigma$ of the full shift $\{0, 1, \cdots, m\}^{\mathbb{Z}}$, and a continuous surjective map $\pi: \Sigma \to X$ such that $f \circ \pi = \pi \circ \sigma$, where $\sigma$ stands for the shift map.

A symbolic extension $(\Sigma, \sigma, \pi)$ is said to be \emph{principal} if $\pi$ preserves the metric entropy, that is, $h_{\eta}(\sigma) = h_{\mu}(f)$ for every $f$-invariant measure $\mu$ and every $\sigma$-invariant measure $\eta$ such that $\mu = \pi_*(\eta)$. If, in addition, there is a Borel measurable map $\tau: X \to \Sigma$ such that $\pi \circ \tau = \text{Identity}_X$, $\,\,\sigma \circ \tau = \tau \circ f$ and $\Sigma=\overline{\tau(X)}$, then $(\Sigma, \sigma, \pi, \tau)$ is called a \emph{symbolic extension with embedding}.

A symbolic extension $(\Sigma, \sigma, \pi)$ is \emph{strongly faithful} if the induced map 
$\pi_*: \mathscr{P} (\Sigma,\sigma)\,\to \,\mathscr{P}(X,f)$  is an homeomorphism and $\pi$ preserves periodic points, that is, for any $n \in \mathbb{N}$ we have $\pi(\mathrm{Per}_n(\sigma_{|\Sigma})) = \mathrm{Per}_n(f)$.

The existence of symbolic extensions seems to depend on hyperbolic-type properties of $f$ and its degree of differentiability. For instance, in the setting of $C^\infty$ diffeomorphisms on a compact Rimannian manifold, J. Buzzi established in \cite{Bz2005} that principal symbolic extensions always exist. On the other hand, D. Burguet proved in \cite{B2011} that, for $C^2$ diffeomorphisms on surfaces, symbolic extensions are sure to exist. On the contrary, T. Downarowicz and S. Newhouse proved in \cite{DN2005} that a generic $C^1$ area-preserving diffeomorphism of a compact surface either is Anosov or has no symbolic extension.

M. Boyle, D. Fiebig and U. Fiebig showed in \cite{BFF2002} that, if $f$ is entropy-expansive, then it has a principal symbolic extension. In addition, W. Cowieson and L.-S. Young proved in \cite{CY2005} that every partially hyperbolic $C^1$ diffeomorphism with a one-dimensional center bundle is entropy-expansive (see generalizations in~\cite{DF2011, DFPV2012}). Therefore, if $f$ is partially hyperbolic with a one-dimensional center bundle then a principal symbolic extension exists. In particular, every Shub's example in $\mathcal{U}$ has a principal symbolic extension. In addition, if we restrict to $\mathfrak{R}$, the diffeomorphisms are asymptotically per-expansive, and we may find a strongly faithtful extension with embedding.


For further use, we register that, according to \cite[Main Theorem]{B2017}, the following four conditions together are enough to guarantee that $f$ has a principal strongly faithful symbolic extension with embedding:
\begin{enumerate}
\item[(1)] $f$ is entropy-expansive.
\smallskip
\item[(2)] $f$ is asymptotically per-expansive.
\smallskip
\item[(3)] $\mathrm{Per}(f)$ is zero dimensional.
\smallskip
\item[(4)] There exists $K > 0$ such that
\begin{enumerate}
\item[(i)] $h_{\mathrm{top}}(f) < \log K$;
\smallskip
\item[(ii)] $\frac{1}{n}\,\log \#\, \mathrm{Per}_n(f)  \leqslant \log K$ for every $n \in \mathbb{N}$.
\end{enumerate}
\end{enumerate}

\subsection{Partial hyperbolicity}\label{ss:partial-hyp}
Assume in the following subsections that $X$ is a compact, connected Riemannian manifold. An $f$-invariant compact set $\Lambda \subset X$ is \emph{partially hyperbolic} if the tangent bundle on $\Lambda$ admits a $Df$-invariant splitting
$E^{\mathrm s}(f) \oplus E^{\mathrm c}(f) \oplus E^{\mathrm u}(f)$ such that $E^{\mathrm s}$ is uniformly contracted and $E^{\mathrm u}$ is uniformly expanded, and the possible contraction and expansion of $Df$ in $E^c(f)$ are weaker than those in the complementary bundles. More precisely, there exist constants $N \in \mathbb{N}$ and $\lambda > 1$ such that, for every $x \in \Lambda$ and every unit vector $v^\ast\in E^{\ast}(x,f)$, where $\ast=\mathrm s,\mathrm c, \mathrm u$, we have
\medskip
\begin{itemize}
\item[(a)] $\quad \lambda\, \|\,Df_x^N (v^{\mathrm s})\,\| <  \|\,Df_x^N (v^{\mathrm c})\,\| < \lambda^{-1}\,\|\,Df_x^N (v^{\mathrm u})\,\|$
\medskip
\item[(b)] $\quad \|\,Df_x^N (v^{\mathrm s})\,\| < \lambda^{-1} < \lambda < \|\,Df_x^N (v^{\mathrm u})\,\|$.
\medskip
\end{itemize}

In particular, an $f$-invariant compact set $\Lambda \subset X$ is said to be a \textit{partially hyperbolic attracting set} if there exists an open neighborhood $U$ of $\Lambda$ such that $\overline{f(U)} \subset U$ and
$$\Lambda=\bigcap_{n\,\in\, \mathbb{N}} \,f^n(U)$$
on which there exist a continuous $Df$-invariant splitting of the tangent bundle into a strong unstable sub-bundle $E^{\,\mathrm{uu}}$ and a center sub-bundle $E^{\,\mathrm c}$ dominated by $E^{\,\mathrm{uu}}$. More precisely,
$T_\Lambda X = E^{\,\mathrm{uu}} \oplus E^{\,\mathrm c}$ and
$$\|\,(Df\,|\,_{E^{\,\mathrm {uu}}})^{-1}\,\| < 1 \quad \quad \mbox{and} \quad \quad
\|\,Df\,|\,_{E^{\,\mathrm{c}}}\,\|\,\,\|\,(Df\,|\,_{E^{\,\mathrm {uu}}})^{-1}\,\| < 1.$$

Partial hyperbolicity is a robust property, and a partially hyperbolic diffeomorphism $f$ admits stable and unstable foliations, say $W^{\mathrm s}(f)$ and $W^{\mathrm u}(f)$, which are $f$-invariant and tangent to $E^{\mathrm s}(f)$ and $E^{\mathrm u}(f)$, respectively \cite{BDV2005}. However, the center bundle $E^{\mathrm c}(f)$ may not have a corresponding tangent foliation (cf. \cite{HHU2012}).
For a comprehensive exposition on partial hyperbolicity, we refer the reader to 
\cite{BDV2005}.

Suppose that $f$ has a partially hyperbolic attracting set. We say that $f$ is \emph{mostly contracting} if, from the point of view of the natural volume within the strong unstable leaves, the asymptotic forward behavior along the central direction is contracting: given any $uu$-dimensional disk $D$ inside a strong unstable leaf of $W^{\,\mathrm{uu}}$, there exists a positive volume measure subset $A \subset D$ whose points satisfy
$$\limsup_{n \, \to \, +\infty}\, \frac{1}{n}\, \log \|\, Df^n\,|\,_{E^{\,\mathrm c}(x)} \,\| < 0 \quad \quad \forall \,\,x \in A.$$
We note that, according to \cite{An2010}, the set of partially hyperbolic diffeomorphisms whose central direction is mostly contracting is open in the $C^k$ topology for any $k > 1$.

\subsection{Hyperbolic measures}\label{ss.hyperbolic-measures}
Given $x \in X$ and $v \in T_x X$, define the \textit{upper Lyapunov exponent of $v$ at $x$} by
\begin{equation*}\label{e.exp}
\lambda^+(x,v) \,\eqdef\,\, \limsup_{n\,\to\,+\infty} \,\,\frac1{n}\,\log \|\,D_xf^n(v)\,\|.
\end{equation*}
The function $\lambda^+: TX \,\to \,\mathbb{R}$ can only take a finite number $r(x)$ of different values on each space $T_x X$, say $\lambda_1(x)<\lambda_2(x)< \cdots<\lambda_{r(x)}(x)$, and associated to these there exist a filtration
$L_1(x) \subset L_2(x) \subset \cdots \subset  L_{r(x)}(x) = T_x X$ such that $\lambda^+(x,v) = \lambda_i(x)$ for every $x \in X$ and all $v \in L_i(x) \setminus L_{i-1}(x)$. Besides, the maps $\big(\lambda_i(x)\big)_{1 \,\leqslant \,i \,\leqslant \,r(x)}$ are measurable and $f$-invariant; their values are called the \textit{Lyapunov exponents of $f$ at $x$}. For each $1 \,\leqslant \,i \,\leqslant \,r(x)$ and $x \in X$, the number $k_i(x)=\dim(L_i(x))-\dim(L_{i-1}(x))$ is the multiplicity of the $i$-th exponent at $x$. Moreover, there exists a subset $\mathscr{O}(f) \subset X$ such that, if $x$ belongs to $\mathscr{O}(f)$, then the limit
$$\lim_{n\,\to\,+\infty} \,\frac1{n}\,\log\,\|\,D_xf^n(v)\,\|$$
exists for all $v\neq 0$. The elements in $\mathscr{O}(f)$ are called \textit{regular points}.  Oseledets' Theorem \cite{O1968} ensures that the set of regular points $\mathscr{O}(f)$ has full $\mu$ measure for any $\mu \in \mathscr{P}(X,f)$. If, in addition, $\mu$ is ergodic, then the functions $x \,\to\, \lambda_i(x)$ and $x \,\to\, r(x)$ are constant at $\mu$-almost everywhere. We denote these constants by $\lambda_1(\mu) < \cdots < \lambda_r(\mu)$. An ergodic probability measure $\mu$ is said to be \textit{hyperbolic} if $\lambda_i(\mu)\neq 0$ for every $i = 1,..., r$.


\subsection{SRB measures}\label{ss.SRB-measures}
Let $x \in X$ be a regular point and consider the sum (with multiplicity) of all the positive Lyapunov exponents at $x$
$$\chi^u(x) \,\eqdef \,\sum\limits_{\{i\,:\,\lambda_i(x)\,>\,0\}}\,k_i(x)\,\lambda_i(x).$$
Margulis-Ruelle inequality \cite{Ru1978} states that the metric entropy (definition in \cite[\S4]{W1981}) of every $\mu\in \mathscr{P}(X,f)$ is bounded above by the space average of $\chi^u$,
that is,
\begin{equation*}\label{e.Ruelle}
h_{\mu}(f) \leqslant \int\,\chi^u\,\,d\mu.
\end{equation*}
On the other hand, by Oseledets' Theorem one knows that, if $E^u(x)$ stands for the subspace of $T_x X$ corresponding to the positive Lyapunov exponents at the regular point $x \in X$ and we denote by $J^u(x)$ the Jacobian of $Df$ restricted to the subspace $E^u(x)$, then
$$\chi^u(x) = \lim_{n\,\to\,+\infty}\,\frac1{n}\,\sum\limits^{n-1}_{i=0}\,\log|\,J^u\big(f^i(x)\big)\,|.$$
Thus, for every Borel $f$-invariant probability measure $\mu$ one has
\begin{equation}\label{e.SRB}
h_{\mu}(f) \leqslant \int\,\log|\,J^u\,|\,d\mu.
\end{equation}
A probability measure $\mu$ attaining the equality in \eqref{e.SRB} is called an \textit{SRB measure}. Pesin proved in \cite{Pes1977} that if $\mu$ is equivalent to Lebesgue measure (the Riemannian volume) then $\mu$ is an SRB measure. Afterwards, Ledrappier and L.-S. Young identified all the measures satisfying Pesin's entropy formula, establishing in \cite{LY1985} that the equality \eqref{e.SRB} holds if and only if the conditional measures of $\mu$ on unstable manifolds are absolutely continuous with respect to Lebesgue measure.

\subsection{Physical measures}\label{ss.physical-measures}
Let $\mu$ be a Borel $f$-invariant probability measure on $X$. A point $x \in X$ is called \emph{$\mu$-generic} if
\begin{equation*}\label{e.fis}
\lim_{n\,\to\,+\infty}\,\frac1{n}\,\sum^{n-1}_{i=0}\,\varphi(f^i(x)) = \int\,\varphi\,d\mu \qquad \forall \,\,\varphi \in C^0(X,\mathbb{R})
\end{equation*}
where $C^0(X,\mathbb{R})$ stands for the space of continuous maps $\varphi: X \,\to\, \mathbb{R}$ with the uniform norm. We will denote by $\mathfrak{B}(\mu)$ the set of $\mu$-generic points, also called the \textit{basin of attraction of} $\mu$. The measure $\mu$ is called \textit{physical} if $\mathfrak{B}(\mu)$ has positive Lebesgue measure. Note that, if
the basin of $\mu$ has full Lebesgue measure, then $\mu$ is the unique physical measure of $f$.

For Anosov diffeomorphisms and, more generally, $C^2$ Axiom A attractors, there exists a unique invariant probability measure $\mu$ which is characterized by each of the following properties, equivalent to one another (cf.  \cite{B1975}):
\begin{enumerate}
\item[(1)] Equality \eqref{e.SRB} holds (that is, $\mu$ is SRB).
\smallskip
\item[(2)] The conditional measures on unstable manifolds of $\mu$ are absolutely continuous with respect to Lebesgue measure.
 \smallskip
\item[(3)] Lebesgue almost every point in a neighborhood of the attractor is generic with respect to $\mu$ (that is, $\mu$ is physical).
\end{enumerate}

\section{The setting}\label{se.Shub-examples}

We now review the construction of Shub's examples \cite{S1971} with the additional constraints imposed in \cite{NY1983}. Let $\Phi \colon  \mathbb{T}^2  \to \mathbb{T}^2 $ be an Anosov diffeomorphism and $\mathrm{T}\,\mathbb{T}^2  = E^{\mathrm{ss}} \oplus E^{\mathrm{uu}}$ be its hyperbolic splitting satisfying, for some uniform constant $0 < \gamma < 1$,
$$\max\,\left\{\|\,D\Phi|_{E^{\mathrm{ss}}}\,\|, \quad \|\,D\Phi^{-1}|_{E^{\mathrm{uu}}}\,\|\right\} < \gamma.$$
Assume that $\Phi$ has two fixed points $p$ and $q$. Note that they are homoclinically related (that is, both intersections $W^{\mathrm s}(p)\pitchfork W^{\mathrm u}(q)$ and $W^{\mathrm u}(p)\pitchfork W^{\mathrm s}(q)$ are transversal and non-empty). Afterwards, take a smooth family of torus $C^1$-diffeomorphisms $f_x \colon \,\mathbb{T}^2 \,\to\, \mathbb{T}^2$ indexed by $x \in \mathbb{T}^2$ with the following properties:
\smallskip
\begin{itemize}
\item[(P1)] At each $x \in \mathbb{T}^2$, the tangent space at $x$ admits a splitting
$\mathrm{T}\,\mathbb{T}^2_x = E^{\mathrm{c}}(f_x) \oplus E^{\mathrm{u}}(f_x)$
invariant under $Df_x$ and for which there exist constants $0 < \gamma_1 < \gamma^{-1}_2<1$ such that
$$\|\,D{f_x^{-1}}|_{E^{\mathrm{u}}(f_x)}\,\|  \leqslant   \gamma_1  \quad \quad \text{and} \quad \quad   \gamma_1   \leqslant
\|\,D f_x|_{E^{\mathrm{c}}(f_x)} \,\| \leqslant \,\gamma_2.$$
We may assume, taking a power of $\Phi$ if necessary, that $\gamma < \gamma_1$.
\medskip
\item[(P2)] For every $x \in \mathbb{T}^2$, the diffeomorphism $f_x$ preserves cone fields $\mathcal{C}^\mathrm{cs}$ and $\mathcal{C}^\mathrm{u}$.
\medskip
\item[(P3)] The map $f_p$ is Anosov, while $f_q$ is a Derived from Anosov.
\medskip
\item[(P4)] There is $\theta_0 \in \mathbb{T}^2$ such that $f_x(\theta_0) = \theta_0$ for every $x$, and $\theta_0$ is a saddle of $f_p$ and a source for $f_q$.
\medskip
\end{itemize}
Shub's examples are precisely the skew products induced by the action of the diffeomorphisms $(f_x)_{x \, \in \, \mathbb{T}^2}$, namely
$F_S \colon \mathbb{T}^2 \times \mathbb{T}^2  \to \mathbb{T}^2\times \mathbb{T}^2 $ defined by
\begin{eqnarray}\label{eq:skew-product}
F_S(x,y)=\big(\Phi(x), \, f_x(y)\big) .
\end{eqnarray}
It is not difficult to check that
$\Omega(F_S) = \mathbb{T}^2 \times \mathbb{T}^2$ and that $F_S$ is partially hyperbolic with a one-dimensional center bundle and a splitting
$$\mathrm{T}_{(x,y)}\big(\mathbb{T}^2\times \mathbb{T}^2\big) =E^{\mathrm{ss}}(x,y)\oplus E^{\mathrm{c}}(x,y)\oplus E^{\mathrm{u}}(x,y)\oplus E^{\mathrm{uu}}(x,y).$$
In what follows we denote by $\mathcal{W}^*$ the invariant foliation tangent to $E^*$, with $*=$ ss, c, u, uu.

\subsection{Additional assumptions}\label{sse.assumptions}

The selection of the family $x \in \mathbb{T}^2\to f_x\in \mathrm{Diff}^1(\mathbb{T}^2)$ in the construction of $F_S$ is not unique. In this work we will add the conditions (A1)-(A3) below, as done in \cite{NY1983}:
\begin{itemize}
\item[(A1)] The map  $x \in \mathbb{T}^2 \to f_x \in \mathrm{Diff}^1(\mathbb{T}^2)$ is continuous.
\medskip
\item[(A2)] $F_S$ is homotopic to $\Phi \times L$ as a bundle map (that is, the homotopic path is made of skew-products with fixed base $\Phi$), where $L: \mathbb{T}^2\to  \mathbb{T}^2$ is a hyperbolic toral automorphism.
\medskip
\item[(A3)] Each $f_x$ preserves the stable foliation $\mathcal{F}$ of $L$ which is tangent to the central direction $E^{\mathrm c}$. We also require that this foliation is normally expanded, meaning that
    $$\inf_{(x,y)\,\in\, \mathbb{T}^2\times\mathbb{T}^2}\,\|\,D_yf_x\,|\,_{E^{\mathrm u}(x,y)}\,\|\,> \,\max  \,\Big\lbrace 1,\sup_{(x,y)\,\in\, \mathbb{T}^2\times\mathbb{T}^2}\,\|\,D_yf_x\,|\,_{E^{\mathrm c}(x,y)}\,\|\Big\rbrace$$
\end{itemize}

\subsection{Construction of Shub's examples}\label{sse.construction}
Although Shub's examples can be constructed quite generally, the natural way to obtain them is through a small $C^0$-perturbation supported on a small neighborhood of a fixed point of $\Phi\times L$. More precisely, suppose that
$0 <\lambda_s <1$ and $\lambda_u =\lambda_s^{-1}>1$
are the eigenvalues associated to the unstable and stable eigenvectors $\textbf{v}^u$ and $\textbf{v}^s$ of the matrix $L$. Let $\theta_0\in \mathbb{T}^2$ be a fixed point of the induced map by $L$ (which we still denote by $L$ if no confusion arises) corresponding to $\textbf{0}$ in $\mathbb{R}^2$. In a relatively small neighborhood $W\, \eqdef\, W_1\times W_2$ of $(q,\theta_0)$ we use coordinates $u_1\textbf{v}^u +u_2\textbf{v}^s$ in each fiber $\{w\}\times W_2$, where $w \in W_1$. Let $\varrho > 0$ be small enough so that the ball $B_{\varrho}(q,\theta_0)=B_{\varrho}(q)\times B_{\varrho}(\theta_0)$ of radius $\varrho$ centered at $(q,\theta_0)$ is contained in $W$. Take a bump function $\delta: \mathbb{T}^2\times\mathbb{T}^2\to \mathbb{R}$ defined by $\delta(x,y)\eqdef b(x)\,b(y)$, where $b: \mathbb{T}^2\to \mathbb{R}$ is a bump function satisfying
$0 \leqslant b(x) \leqslant 1$ for every $x \in \mathbb{T}^2$, $\,b(x)=1$ if $|x|<\varrho/2$ and $b(x)=0$ if $|x|>\varrho$. Afterwards consider the system of differential equations in $\mathbb{T}^2\times \mathbb{T}^2$ given by
\begin{equation}\label{e.EDO}
\left\{\begin{array}{lr}
\dot{w}=0 \quad \mbox{in}\quad \mathbb{T}^2  & \smallskip\\
(\dot{u}_1,\,\dot{u}_2) = \big(0, \,\,u_2\,\delta \,(|w-q|,\,|(u_1,u_2)|)\big) \quad \mbox{in}\quad \mathbb{T}^2  &
\end{array}
\right.
\end{equation}
Let $\varphi^t$ be the flow of the differential equation \eqref{e.EDO}, that is,
\begin{equation*}
\varphi^t(w,\,(u_1,u_2)) = \big(w,\psi^t_w(u_1,u_2)\big)\qquad\mbox{where}\qquad
\psi^t_w(u_1,u_2) = \big(u_1,\psi^t_{w,\,2}(u_1,u_2)\big).
\end{equation*}
Then the support of $\varphi^t - id$ is contained in $W$. Moreover, the derivative of the flow at $(w,\theta_0)$ in terms of the $(w,u_1,u_2)$-coordinates is given by
\begin{equation*}
D_{(w,\,\theta_0)}\,\varphi^t=
\begin{pmatrix}
\textbf{1} &  \textbf{0}\\
\textbf{0} & D_{\theta_0}\,\psi^t_w
\end{pmatrix}
\qquad\mbox{where}
\qquad
D_{\theta_0}\,\psi^t_w=
\begin{pmatrix}
1 &  0\\
0 & e^{t\,b(|w-q|)}
\end{pmatrix}
\end{equation*}
where the bold numbers $\textbf{0}$ and  $\textbf{1}$ stand for the null $2\times 2$ matrix and the $2\times 2$ identity matrix, respectively.

Finally, fix $T>0$ such that $1 <\lambda_s\, e^{T} < \lambda_u$ and define $F_S: \mathbb{T}^2\times \mathbb{T}^2\to \mathbb{T}^2\times \mathbb{T}^2$ by
$$F_S \,\eqdef\, \varphi^T \circ (\Phi\times L).$$
This way the derivative of $F_S$ at $(q,\theta_0)$ in the $(w,u_1,u_2)$ coordinate system is precisely
$$D_{(q,\,\theta_0)}\,F_S = D_{(q,\,\theta_0)}\,\varphi^T\,\, D_{(q,\,\theta_0)}(\Phi \times L)=
\begin{pmatrix}
\textbf{1} &  \textbf{0}\\
\textbf{0} & D_{\theta_0}\,\psi^T_q
\end{pmatrix}
\begin{pmatrix}
D_q\Phi &  \textbf{0}\\
\textbf{0} & L
\end{pmatrix}
=
\begin{pmatrix}
D_q\Phi &  \textbf{0}\\
\textbf{0} & D_{\theta_0}\,\psi^T_q\, L
\end{pmatrix}
$$
where
$$D_{\theta_0}\,\psi^T_q\, L =
\begin{pmatrix}
\lambda_u &  0\\
0 & \lambda_s\, e^{T}
\end{pmatrix}.$$
Therefore, $(q,\,\theta_0)$ is a fixed point of unstable index $3$, while $(p,\,\theta_0)$ has unstable index $2$. Furthermore, the one-parameter family $(f_x)_{x \, \in \, \mathbb{T}^2}$ is given by
$$f_x = \psi^{T}_{\Phi(x)}\circ L.$$


\begin{figure}
\centering
\begin{overpic}[scale=.12,
  ]{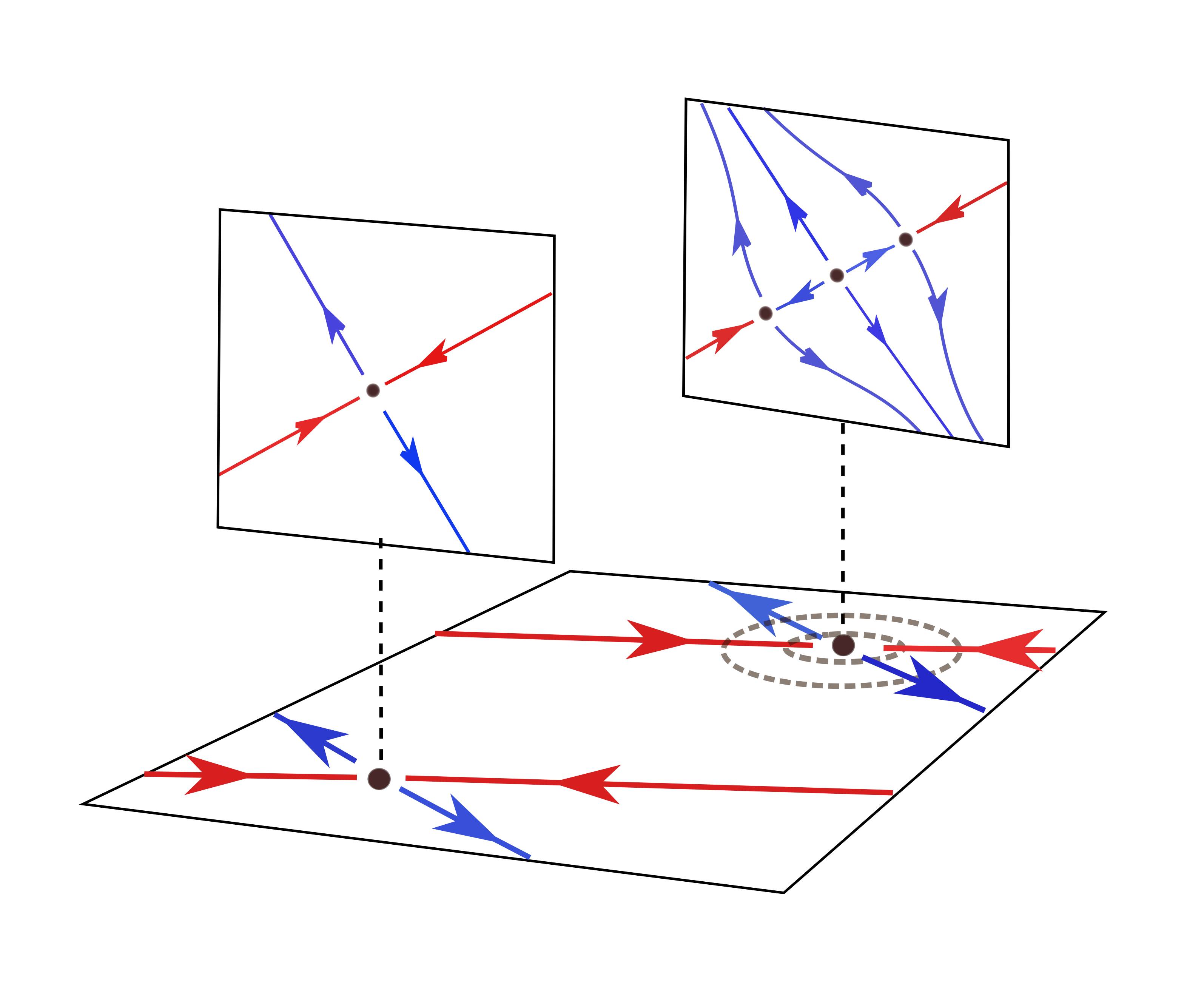}
            \put(26,68){\large{Anosov}}
           \put(31,53){\large{$\theta$}}
            \put(34,20){\LARGE{$p$}}
     \put(67,23){\LARGE{$q$}}
           \put(69,61){\large{$\theta$}}
            \put(55,77){\large{Derived from Anosov}}
    \end{overpic}
\caption{Homotopic deformation from $\Phi\times L$ to $F_S$}
\label{fig:Sh}
\end{figure}
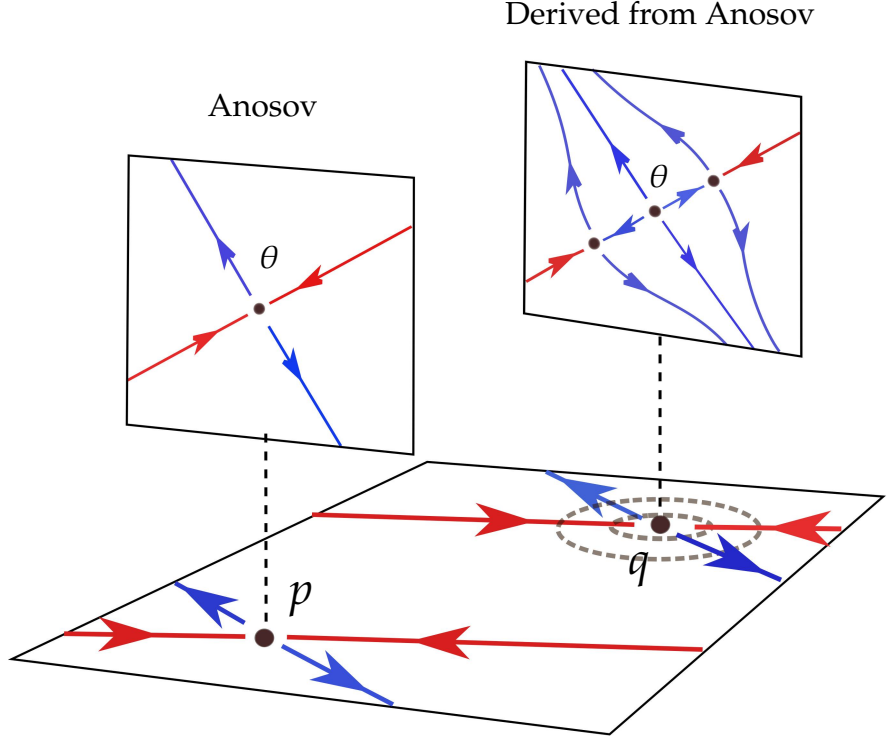

\begin{rem}\label{re.Kupka-Smale}
The previous construction provides an open set $\mathcal{U}$ of $C^r$ diffeomorphisms, $r > 1$, with the properties (P1)-(P4) and (A1)-(A3) listed above. Indeed, the conditions (P1)-(P4) are valid for $\gamma_1:= \lambda_s$ and
$\gamma_2 := e^{T}\,\lambda_s$;
and (A1)-(A3) are obtained by construction.
\end{rem}

\section{Properties of $F_S \in \mathcal{U}$}\label{se.properties}

The selection of the family $x \in \mathbb{T}^2\to f_x\in \mathrm{Diff}(\mathbb{T}^2)$ in the previous construction of $F_S$ induce several dynamical and ergodic features in $F_S$ we will now list.

\subsection{Semi-conjugation with an Anosov diffeomorphism}\label{sse.semi-conjugation}

Under the previous assumptions, it was shown in \cite{NY1983} the existence of a continuous surjective map $H: \mathbb{T}^2 \times \mathbb{T}^2 \, \to \mathbb{T}^2 \times \mathbb{T}^2$ such that
\begin{equation}\label{e.rubik}
H\circ F_S = (\Phi \times L)\circ H.
\end{equation}
Moreover, the semi-conjugation $H$ is a skew product as well, that is,
$$H(x,y)=(x,\,h_x(y))$$
where each $h_x :\{x\}\times \mathbb{T}^2\to \{x\}\times \mathbb{T}^2$ is homotopic to the identity 
and satisfies
\begin{equation}\label{e.semi}
h_{\Phi(x)} \circ f_x= L\circ h_x, \quad\quad \forall \, x \in \mathbb{T}^2.
\end{equation}
The semi-conjugation $H$ can be seen as the result of a parameterized version of a theorem due to Franks \cite{Fr}.

In \cite{NY1983}, Newhouse and L.-S. Young also established the existence of a unique probability measure $\mu_{\mathrm{max}}$ of maximal entropy for $F_S$, and proved that $H_*(\mu_{\mathrm{max}})=\nu_{\mathrm{max}}$, where $\nu_{\mathrm{max}}$ stands for the probability measure of maximal entropy of $\Phi\times L$. Moreover, the pairs $(F_S,\mu_{\mathrm{max}})$ and $(\Phi \times L,\nu_{\mathrm{max}})$ are almost conjugate. More precisely, there exists a set $B \times \mathbb{T}^2$ contained in the set of injectivity points of $H$, say
$$\mathcal{A} \,\eqdef\,\Big\{\,\, (x,y)\in \mathbb{T}^2\times \mathbb{T}^2 \,\,  \colon \,\,  \#\,H^{-1}(x,y) = 1 \,\,  \Big\}$$
and such that
\begin{itemize}
\item $\mu_{\mathrm{max}}(B\times \mathbb{T}^2)=\nu_{\mathrm{max}}(B\times \mathbb{T}^2)= 1$;
\smallskip
\item $H:B\times \mathbb{T}^2\to B\times \mathbb{T}^2$ is a conjugation between the restrictions $F_S|_{B\times \mathbb{T}^2}$ and $(\Phi \times L)|_{B\times \mathbb{T}^2}$.
\end{itemize}
Actually, the set $B\times \mathbb{T}^2$ is contained in
$$\mathcal{E}\,\eqdef\, \Big\{ \,\,(x,y)\in \mathbb{T}^2\times \mathbb{T}^2 \,\, \colon \,\,  \lambda^c_+(F_S)(x,y) < 0\,\, \Big\} \subset \mathcal{A}$$
(cf. \cite{NY1983}), where $\lambda^c_+(F_S)$ stands for the upper Lyapunov exponent of $F_S$ along to the one-dimension central direction $E^{\mathrm{c}}(F_S)$.

The following proposition lists some useful consequences from the existence of the semi-conjugation $H$ with the previous properties.

\begin{prop}\cite[Theorems 1 \& 2]{NY1983}\label{prop.Shub-NY} $\,$
\begin{enumerate}
\item[(a)] $h_{\mathrm{top}}(F_S) = h_{\mathrm{top}}(\Phi\times L) =  h_{\mathrm{top}}(\Phi) + h_{\mathrm{top}}(L)$.
\medskip
\item[(b)] $h_{\mathrm{top}}\big(H^{-1}(x,y)\big)=0, \quad \quad \forall \,\,(x,y)\,\in\,\mathbb{T}^2\times \mathbb{T}^2$.
\medskip
\item[(c)] $\mu_{\mathrm{max}}\big(\, \mathcal{E}\,\big)=\nu_{\mathrm{max}}\big(\, \mathcal{E}\,\big)=1$.
\end{enumerate}
\end{prop}

\medskip

We observe that $\nu_{\mathrm{max}}$ is a product measure and that, when $\Phi$ is a linear hyperbolic automorphism, $\nu_{\mathrm{max}}$ is Lebesgue measure on $\mathbb{T}^2\times\mathbb{T}^2$. In this case, Proposition~\ref{prop.Shub-NY} (c) indicates that $F_S$ is mostly contracting along the central direction with respect to the splitting
$$\mathbb{E}^{\,\mathrm{c}} = E^{\,\mathrm{ss}} \oplus E^{\,\mathrm{c}} \quad \text{and} \quad \mathbb{E}^{\,\mathrm{u}} = E^{\,\mathrm{u}}\oplus E^{\,\mathrm{uu}}$$
since Lebesgue almost every point in $\mathbb{T}^2\times \mathbb{T}^2$ has negative central Lyapunov exponent.

Taking into account that $\mathcal{E} \subset \mathcal{A}$, the previous properties (b) and (c) of Proposition~\ref{prop.Shub-NY} allow us to apply \cite[Theorem 1.5]{BFSV2012} to $F_S$, and thereby show that $\mu_{\mathrm{max}}$ describes the \textit{distribution of periodic classes} of $F_S$. More precisely, consider the equivalence relation on the set $\mathbb{T}^2\times \mathbb{T}^2$
$$(x,y) \sim (x_0, y_0) \quad \quad \Leftrightarrow \quad \quad H(x,y) = H(x_0,y_0)$$
so the elements in the class $[(x,y)]$ are the ones in $H^{-1}(H(x,y))$. The class $[(x,y)]$ is said to be $n$-periodic if $H(x,y)$ belongs to $\mathrm{Per}_n(\Phi \times L)$. Denote by $\widetilde{\mathrm{Per}_n}(F_S)$ the set of periodic classes with period $n$. Then $\mu_{\mathrm{max}}$ describes the distribution of periodic classes of $F_S$ if $\mu_{\mathrm{max}}$ is the weak$^*$ limit of the sequence of measures
\begin{equation*}\label{eq:zeta}
n \in \mathbb{N} \quad \mapsto \quad \zeta_n \,\eqdef\, \frac1{\#\,\widetilde{\mathrm{Per}}_n(F_S)}
 \sum_{[(x,y)]\,\in\, \widetilde{\mathrm{Per}}_n(F_S)}\,\delta_{[(x,y)]}
 \end{equation*}
where $\delta_{[(x,y)]}$ is any $F_S^n$-invariant probability measure supported on the class $[(x,y)]$. We may wonder whether the elements of a periodic class are $F_S$-periodic or if the class contains a periodic point by $F_S$. We will answer to this question on Section~\ref{se.proof-Theorem-A}, where we will also establish that our assumptions about $F_S$ enable us to improve \cite[Theorem 1.5]{BFSV2012}, showing that the measure $\mu_{\mathrm{max}}$ describes the distribution of the \textit{periodic points} of $F_S$ (cf. Subsection~\ref{prop.weak-limit}).

\subsection{Hyperbolicity of $\mu_{\mathrm{max}}$}\label{sse.hyperbolicity}

A direct consequence of the construction of $F_S$ is the fact that the measure $\mu_{\mathrm{max}}$ exhibits four Lyapunov exponents, namely
$$\lambda^{ss}(\mu_{\mathrm{max}}) < \lambda^{c}(\mu_{\mathrm{max}}) < \lambda^{u}(\mu_{\mathrm{max}}) < \lambda^{uu}(\mu_{\mathrm{max}})$$
which are constant $\mu_{\mathrm{max}}$ almost everywhere (since $\mu_{\mathrm{max}}$ is ergodic) and satisfy
$$\lambda^{ss}(\mu_{\mathrm{max}}) < 0 < \lambda^u(\mu_{\mathrm{max}}) \quad \quad \text{and} \quad \quad 
\log(\gamma_1) \leqslant \lambda^{c}(\mu_{\mathrm{max}}) \leqslant  \log (\gamma_2).$$
Therefore:

\begin{prop}\label{prop.minimal} $\,$
\begin{enumerate}
\item[(a)] $\mu_{\mathrm{max}}$ is hyperbolic.
\medskip
\item[(b)] The unstable foliation of $F_S$ tangent to the bundle $E^{\mathrm{u}}\oplus E^{\mathrm{uu}}$  is robustly minimal.
\medskip
\item[(c)] If $\Phi$ is a linear hyperbolic automorphism of the $2$-torus, the partially hyperbolic attractor $\mathbb{T}^2\times \mathbb{T}^2$ of $F_S$ admits a unique SRB measure, say $\mu_{\mathrm{SRB}}$, whose basin has full Lebesgue measure. So $\mu_{\mathrm{SRB}}$ is the unique physical measure of $F_S$.
\end{enumerate}
\end{prop}

\begin{proof}
To prove the  hyperbolicity of $\mu_{\mathrm{max}}$ we need to verify that $\lambda^{c}(\mu_{\mathrm{max}}) \neq 0$. Let $\mathscr{O}(F_S)$ be the set of regular points of $F_S$ provided by Oseledets's Theorem \cite{O1968}. From Proposition~\ref{prop.Shub-NY} (c), we know that the set $\mathscr{O}(F_S)\, \cap\, \mathcal{E}$ has full $\mu_{\mathrm{max}}$ measure. Moreover, points $(x,y)$ in this intersection
  satisfy
$$
\lambda^c_-(F_S) (x,y)=\lambda^c_+(F_S) (x,y)
=
\lambda^{c}(\mu_{\mathrm{max}}) < 0$$
where $\lambda^c_-(F_S) (x,y)$ denotes the lower Lyapunov exponent of $F_S$ at $E^c(x,y)$.
This completes the proof of item (a).

The robust minimality, asserted in item (b), of the unstable foliation $W^{\mathrm u}$ of $F_S$, tangent to $E^{\mathrm{u}}\oplus E^{\mathrm{uu}}$, follows directly from \cite[Theorem A]{PujSam2006}.

Regarding item (c), we start noticing that, being a particular case to which we may apply the results of \cite{CY2005}, the skew product $F_S$ has at least one SRB measure. As mentioned before, $\mathbb{T}^2\times \mathbb{T}^2$ is a partially hyperbolic attractor for $F_S$ with a partially hyperbolic splitting $\mathbb{E}^c = E^{ss} \oplus E^c$ and $\mathbb{E}^{u}=E^u \oplus E^{uu}$. Under the additional assumption that $\Phi$ is a linear hyperbolic automorphism of the $2$-torus, we know that $\nu_{\mathrm{max}}$ is the Lebesgue measure on $\mathbb{T}^2\times \mathbb{T}^2$. Therefore, from Proposition~\ref{prop.Shub-NY} (c) (which says that $F_S$ is mostly contracting) and the previous item (b) (which states that the unstable foliation of $F_S$ is robustly minimal), we conclude that $F_S$ satisfies the hypothesis of \cite[Theorem B]{BV2000}. Therefore, $F_S$ has a unique SRB measure $\mu_{\mathrm{SRB}}$ whose basin has full Lebesgue measure. Hence, $\mu_{\mathrm{SRB}}$ is the unique physical measure of $F_S$ as well.
\end{proof}

From Proposition~\ref{prop.Shub-NY} (a) we know that $h_{\mu_{\mathrm{max}}}(F_S)> 0$.
 Then we may apply the generalization of~\cite[Corollary 4.3]{K1980} established in~\cite[Theorem 1]{G2016}. This result together with item (a) of Proposition~\ref{prop.minimal}  imply that $F_S$ is a limit of horseshoes in the sense of the entropy.

\begin{prop}\label{prop.horseshoes}
For every $0 < \varepsilon < h_{\mathrm{top}}(F_S)$, there exists a compact $F_S$-invariant set $\Lambda_\varepsilon \subset \mathbb{T}^2 \times \mathbb{T}^2$ such that $F_S|_{\Lambda_\varepsilon}$ is conjugate to a subshift of finite type and
\begin{equation*}\label{eq:entropy}
h_{\mathrm{top}}(F_S|_{\Lambda_\varepsilon}) \geqslant h_{\mathrm{top}}(F_S)-\varepsilon.
\end{equation*}
\end{prop}

\subsection{SRB measures of $F_S$ and $\Phi \times L$}\label{sse.SRB} Given $(x,y)$ in $\mathbb{T}^2 \times \mathbb{T}^2$, let $J^u_{F_S}(x,y)$ be the Jacobian of $D_{(x,y)}F_S$ restricted to the unstable bundle $E^u(x,y)\oplus E^{uu}(x,y)$ of $F_S$. Analogously, define $J^u_{\Phi\times L}(x,y)$.

\begin{prop}\label{prop.SRB} Let $\mu_{\mathrm{SRB}}$ be the SRB measure of $F_S$. Suppose that at $\mu_{\mathrm{SRB}}$ almost every $(x,y)$ in $\mathbb{T}^2 \times \mathbb{T}^2$ we have
\begin{equation}\label{e.eq1}
\vert	\,J^u_{\Phi\times L}\circ H(x,y)\,\vert \leqslant \vert\, J^u_{F_S}(x,y)\,\vert.
\end{equation}
Then $H_*(\mu_{\mathrm{SRB}})$ is the SRB measure of $\Phi \times L$.
\end{prop}

\begin{proof} Set $\nu=H_*(\mu_{\mathrm{SRB}})$. After Margulis-Ruelle inequality \eqref{e.SRB}, we are left to verify that
$$\int\,\log \,\vert	\, J^u_{\Phi\times L}\,\vert	\,d\nu \leqslant h_{\nu}(\Phi\times L).$$
Firstly, we note that
$$h_{\mu_{\mathrm{SRB}}}(F_S) = h_{\nu}(\Phi\times L).$$
Indeed, Proposition~\ref{prop.Shub-NY} (b) and Ledrappier-Walters' formula \cite[(1.2)]{LW1977} yield
$$h_{\mu_{\mathrm{SRB}}}(F_S) \leqslant  h_{\nu}(\Phi\times L)$$
which, together with the well-known fact \cite[Theorem 4.11]{W1981} that $h_{\mu_{\mathrm{SRB}}}(F_S) \geqslant h_{\nu}(\Phi\times L)$, imply the equality. Thus, using~\eqref{e.eq1} one gets
$$\int\,\log \,\vert	\, J^u_{\Phi\times L}\,\vert	\,d\nu  =  \int\,\log \,\vert \,J^u_{\Phi\times L}\circ H \,\vert\, d\mu_{\mathrm{SRB}}
 \leqslant \int\,\log \,\vert \,J^u_{F_{S}}\,\vert \, d\mu_{\mathrm{SRB}}  = h_{\mu_{\mathrm{SRB}}}(F_S) = h_{\nu}(\Phi\times L).$$
\end{proof}

\section{Proof of Theorem~\ref{teo.A}}\label{se.proof-Theorem-A}

As the argument on this section is rather long, for the reader's convenience we subdivide it according to the items of the theorem's statement.

We consider the residual subset $\mathfrak{R}\subset \mathcal{U}$, defined by $\mathfrak{R}\eqdef \mathcal{U}\cap \mathcal{K}\mathcal{S}$, where $\mathcal{K}\mathcal{S}$ denotes the set of Kupka-Smale diffeomorphisms. Thus, its elements have only hyperbolic periodic points, and so the whole set of the periodic points is countable. In what follows, we will consider such generic Shub's examples in $\mathfrak{R}$.

\subsection{Proof of Theorem~\ref{teo.A} (a)}

The key idea in the following argument is to assess how many periodic orbits in $\mathrm{Per}_n(F_S)$, for $n\in\mathbb{N}$, may occur in the intersection $H^{-1}(P)\cap \mathrm{Per}_n(F_S)$ for each periodic point $P\in \mathrm{Per}_n(\Phi\times L)$.

\subsection{Periodic classes}

Consider a foliation $\mathcal{W}$ of a simply connected compact Riemannian manifold $M$ and lift it to the universal cover $\widetilde{M}$, obtaining a foliation we denote by $\widetilde{\mathcal{W}}$. For points $x, y$ on the same leaf $\widetilde{W}$ of $\widetilde{\mathcal{W}}$, one can define a distance $\mathcal{D}_{\widetilde{W}}(x,y)$ as the length of the shortest path inside the leaf $\widetilde{W}$ linking $x$ and $y$. We say that the lifted foliation $\widetilde{\mathcal{W}}$ of $\mathcal{W}$ is \textit{quasi-isometric} if there is a constant $C > 1$ such that for any $x,y \in \widetilde M$ lying on the same leaf of $\widetilde{\mathcal{W}}$ we have
$$\mathcal{D}_{\widetilde{W}}(x,y) < C\,\mathcal{D}(x,y) + C$$
where $\mathcal{D}$ denotes the metric on $\widetilde M$.
%

\begin{claim}\label{cl.quasiquasi} Let $\mathcal{W}^u$ and $\mathcal F$ the foliations introduced on Sections \ref{se.Shub-examples} and \ref{sse.assumptions}, respectively. Then
$\widetilde{\mathcal{W}}^u$ and $\widetilde{\mathcal F}$ are quasi-isometric.
\end{claim}

\begin{proof}
Since we wish to estimate the intrinsic distance between two points of the same leaf of either $\widetilde{\mathcal{W}}^u$ or $\widetilde{\mathcal F}$, which is contained in some fiber $\{\tilde x\}\times \mathbb{R}^2$ with $\tilde x\in \mathbb{R}^2$, it is sufficient to consider the lift of $\mathcal{W}^u$ and $\mathcal F$, which we still denote by $\widetilde{\mathcal{W}}^u$ and $\widetilde{\mathcal F}$, to the universal cover $\mathbb{T}^2\times \mathbb{R}^2$ of $\mathbb{T}^2\times \mathbb{T}^2$ with respect to the second factor.

Firstly we observe that from \cite[Lemma 4.A.5]{Pot} we know that, for each $x \in \mathbb{T}^2$, the foliations $\widetilde{\mathcal{W}}^u(x,\cdot)$ and $\widetilde{\mathcal F}(x,\cdot)$ inside $\{x\}\times\mathbb{R}^2$ have a global product structure. Then the fact that $\widetilde{\mathcal{W}}^u(x,\cdot)$ and $\widetilde{\mathcal F}(x,\cdot)$ are quasi-isometric follows from \cite[Proposition 4.3.9]{Pot}
(see also the proof of \cite[Lemma 4.A.5]{Pot} for more details).
Indeed, this result informs that for every $x \in \mathbb{T}^2$ there exist $C_{1,x}, C_{2,x} > 1$ such that, for every $\tilde{y},\tilde{z}$ in $\mathbb{R}^2$ one has
$$\mathcal{D}_{\widetilde{W}^u}\big((x,\tilde y),(x,\tilde z)\big) < C_{1,x}\,\|\,\tilde{y}-\tilde{z}\,\| + C_{1,x}
\quad\mbox{and}\quad
\mathcal{D}_{\widetilde{\mathcal F}}\big((x,\tilde y),(x,\tilde z)\big) < C_{2,x}\,\|\,\tilde{y}-\tilde{z}\,\| + C_{2,x}.
$$
Moreover, $C_{1,x}$ and $C_{2,x}$ can be taken independent of $x$. On the one hand, note that $C_{1,x} \leqslant C_{1,q}$ for every $x \in \mathbb{T}^2$; on the other hand, the foliation $\widetilde{\mathcal F}$ consists of the family of lines obtained by the translation (in $\mathbb{R}^2$) of the stable manifold of $L$, thus we can take $C_{2,x} = 1$ for every $x$. This completes the proof.
\end{proof}

The next result may be thought of as parameterized version of \cite[Proposition 3.1]{U2012}.

\begin{lem}\label{l.p-classes} For all $(x,y)\in \mathbb{T}^2\times \mathbb{T}^2$, the set $H^{-1}(x,y)$ is a one-dimensional compact connected subset of a center manifold of $F_S$.
\end{lem}

\noindent \emph{Proof.}
The equality~\eqref{e.rubik} can be expressed in $\mathbb{T}^2\times \mathbb{R}^2$ by lifting \eqref{e.semi} to $\{x\} \times \mathbb{R}^2$,
which provides the equality $\widetilde{H}\circ \widetilde{F_S} = (\Phi \times \widetilde{L})\circ \widetilde{H}$, where $\widetilde{H}(x,\tilde{y})=\big(x,\widetilde{h}_x(\tilde{y})\big)$ is a proper map at a bounded distance from the identity map. The former property of $\widetilde{H}$ implies that $\widetilde{h}_x^{-1}(\tilde{y})$ is a compact subset of $\mathbb{R}^2$ for every $(x,\tilde{y})\in\mathbb{T}^2\times \mathbb{R}^2$. The latter leads to the following estimate:
for every $x \in \mathbb{T}^2$ and $\tilde  y,\,\tilde z \in \mathbb{R}^2$,
\begin{equation}\label{e.pascoa}
\widetilde{h}_x(\tilde y)=\widetilde{h}_x(\tilde z) \qquad \Leftrightarrow
\qquad  \exists \,\,\ C > 0: \,\, \|\,\widetilde{F_S}^n(x,\tilde y)-\widetilde{F_S}^n(x,\tilde z)\,\|< C \quad \, \forall n \in \mathbb{Z}.
\end{equation}
Besides, if $\widetilde{W}^s_{\Phi \times L}$ stands for the lifts of the weak stable foliation of $\Phi \times L$ to $\mathbb{T}^2 \times \mathbb{R}^2$, then (cf. \cite[Lemma 2]{NY1983})
$$\widetilde{h}_x\,\big(\widetilde{\mathcal{F}}(x,\tilde{y})\big)=\widetilde{W}^s_{\Phi \times L}\,\big(\widetilde{H}(x,\tilde{y})\big).$$
We are left to verify that $\widetilde{h}_x^{-1}(\tilde{y})$ is a connected set. To do it we show a parametrized version of \cite[Lemma 3.2]{U2012}.

\begin{claim}\label{cl.16}
If $\widetilde{h}_x(\tilde y)=\widetilde{h}_x(\tilde z)$, then $(x,\tilde z) \in \widetilde{\mathcal{F}}(x,\tilde{y})$.
\end{claim}

\begin{proof} Suppose that $(x,\tilde z) \notin \widetilde{\mathcal{F}}(x,\tilde{y})$. Let $(x,\tilde w) = \widetilde{W}^u(x,\tilde z)\cap   \widetilde{\mathcal{F}}(x,\tilde{y})$. Note that such a point $(x,\tilde w)$ exists and is unique (cf. \cite[Proposition 2.4]{HamPot}). Consider
 $$D_{c} = \mathcal{D}_{\widetilde{\mathcal{F}}}\big((x,\tilde y),(x,\tilde w)\big)\quad\mbox{and}\quad
 D_{u} = \mathcal{D}_{\widetilde{\mathcal{W}}^u}\big((x,\tilde z),(x,\tilde w)\big).$$
 Now,
using the parameters $0<\gamma_1<\gamma_2^{-1}<1$ associated to the partial hyperbolicity of $F_S$ (see Section~\ref{se.Shub-examples}), we can find constants  $0<\widetilde{\gamma}_1<\widetilde{\gamma}\,^{-1}_2<1$  such that
$$\|\,\,\widetilde{F_S}^n(x,\tilde y)-\widetilde{F_S}^n(x,\tilde w)\,\,\|\leqslant \widetilde{\gamma}_2\,^n\,\,D_{c}\qquad
\mbox{and} \qquad  \mathcal{D}_{\widetilde{\mathcal{W}}^u}\big(\widetilde{F_S}^n(x,\tilde z),\widetilde{F_S}^n(x,\tilde w)\big)\geqslant \widetilde{\gamma}\,^{-n}_1\,\,D_{u}.$$
Since $\widetilde{W}^u$ is quasi-isometric (Claim~\ref{cl.quasiquasi}), we also have
$$\|\,\,\widetilde{F_S}^n(x,\tilde z)-\widetilde{F_S}^n(x,\tilde w)\,\,\|\geqslant \frac1{C}\big( \widetilde{\gamma}\,^{-n}_1\,\,D_{u}-C\big).$$
Therefore,
$$\|\,\widetilde{F_S}^n(x,\tilde y)-\widetilde{F_S}^n(x,\tilde z)\,\| > \frac1{C}\big( \widetilde{\gamma}\,^{-n}_1\,\,D_{u}-C\big) -  \widetilde{\gamma}\,^{n}_2\,\,D_{c}$$
The last quantity goes to infinity as $n \to +\infty$, which implies, by \eqref{e.pascoa}, that $\widetilde{h} _x(\tilde{y}) \neq  \widetilde{h} _x(\tilde{z})$. This finishes the proof of the claim.
\end{proof}

\begin{claim}\label{cl.connected}
For every $x\in \mathbb{T}^2$ and $\tilde{y}\in \mathbb{R}^2$, the pre-image $\widetilde{h}_x^{-1}(\tilde{y})$ is connected.
\end{claim}

\begin{proof}
We will see that given $\tilde z$ and $\tilde w$ in $\widetilde{h}^{-1}_x(\tilde y)$ then the arc in the center manifold joining $\tilde z$ and $\tilde w$ is contained in $\widetilde{h}^{-1}_x(\tilde y)$. Let $\tilde{\vartheta}$ be a point in this arc. From \eqref{e.pascoa}, we know that $\|\,\widetilde{F_S}^n(x,\tilde z)-\widetilde{F_S}^n(x,\tilde w)\,\| < C$ for every $n \in \mathbb{Z}$. On the other hand, by Claim~\ref{cl.quasiquasi} we have, for every $n$ in $\mathbb{Z}$,
\begin{equation*}
\begin{split}
\|\,\widetilde{F_S}^n(x,\tilde z)-\widetilde{F_S}^n(x,\tilde{\vartheta})\,\|
&\leqslant
\mathcal{D}_{\widetilde{\mathcal{F}}}\big(\widetilde{F_S}^n(x,\tilde z),\widetilde{F_S}^n(x,\tilde{\vartheta})\big)
\\
&\leqslant
\mathcal{D}_{\widetilde{\mathcal{F}}}\big(\widetilde{F_S}^n(x,\tilde z),\widetilde{F_S}^n(x,\tilde{y})\big)
\\
&\leqslant
C+1.
\end{split}
\end{equation*}
Therefore, $\tilde{\vartheta}$ belongs to $\widetilde{h}_x^{-1}(\tilde{y})$. By projecting, the same property is valid for the map $h_x$. This ends the proof of the claim and of Lemma~\ref{l.p-classes}.
\end{proof}

\begin{cor}
For every $n\in \mathbb{N}$ and  $(x,y)\in \mathrm{Per}_n(\Phi\times L)$, the interval $H^{-1}(x,y)$ intersect the set $\mathrm{Per}_n(F_S)$
in at least one point. Therefore,
$$\#\,\mathrm{Per}_n(\Phi\times L) \,\leqslant\, \#\,\mathrm{Per}_n(F_S)  \quad \,\,\forall \,n \in \mathbb{N}.$$
\end{cor}

\begin{proof}
By Lemma~\ref{l.p-classes}, for every $(x,y)\in \mathrm{Per}_n(\sigma\times L)$ the map $F_S^n: H^{-1}(x,y) \to H^{-1}(x,y)$ is a homeomorphism of a closed (possibly degenerate) interval.
Therefore, Brouwer's Fixed Point Theorem guarantees the existence of a fixed point for $F^n|_{H^{-1}(x,y)}$ for every $(x,y)\in \mathrm{Per}_n(\sigma\times L)$.
\end{proof}

\subsection{Cardinality of the periodic orbits}

Recall from Section~\ref{se.Shub-examples} that $F_S$ is a skew product defined by
$$F_S(x,y)= \Big(\Phi(x), \, f_x(y)\Big),  \quad \quad  (x,y) \in \mathbb{T}^2 \times \mathbb{T}^2$$
so, for every $n \in \mathbb{N}$,
$$F^n(x,y)=\big(\Phi^n(x), g^n_{x}(y)\big)$$
where $g^n_{x} \colon \,\mathbb{T}^2 \to \mathbb{T}^2$ is defined by
$$g^n_{x}(y) \,\eqdef\, f_{\Phi^{n-1}(x)}\circ f_{\Phi^{n-2}(x)}\circ \cdots \circ f_{x}(y),  \quad \quad \quad   y\in \mathbb{T}^2.
$$

\begin{prop}\label{prop.DA} Take  $n \in \mathbb{N}$ and $x_0 \in \mathrm{Per}_n(\Phi)$. Then
either, $g^n_{x_0}$ is a Anosov diffeomorphism (conjugated to $L^n$) or a Derived from Anosov (obtained from $L^n$).
\end{prop}

\begin{proof}
Firstly, note that $g^n_{x_0}$ and $L^n$ are semi-conjugated. Indeed, as $x_0 \in \mathrm{Per}_n(\Phi)$ then $h_{\Phi^n(x_0)}=h_{x_0}$ (see~\eqref{e.semi}) and so we have for every $y\in \mathbb{T}^2$
\begin{equation*}
\begin{split}
h_{x_0}\circ g^n_{x_0}(y) & =
h_{\Phi^{n}(x_0)} \circ f_{\Phi^{n-1}(x_0)}  \circ g^{n-1}_{x_0}(y)
=
 L \circ h_{\Phi^{n-1}(x_0)} \circ  \circ g^{n-1}_{x_0}(y)
 =\,\ \cdots\,\, = L^n \circ h_{x_0}(y).
\end{split}
\end{equation*}
Thus, if for every $y\in \mathbb{T}^2$, the interval $H^{-1}(x_0,y)=(x_0,h_{x_0}^{-1}(y))$ is a point, then $y\to H(x_0,y)$ is a conjugation between $g^n_{x_0}$ and  $L^n$, and so $g^n_{x_0}$ is an Anosov diffeomorphism. The remaining case is dealt with on the next lemma.

\begin{lem}\label{l.DA} Take $n \in \mathbb{N}$ and $x_0 \in \mathrm{Per}_n(\Phi)$. If for some $y \in \mathbb{T}^2$ the interval $H^{-1}(x_0,y)$ is non-degenerate, then the diffeomorphism $g^n_{x_0}$ is a Derived from Anosov obtained from $L^n$.
\end{lem}

\noindent \emph{Proof.}
To check that $g^n_{x_0}$ satisfies the standard properties of a Derived from Anosov we will follow the reference~\cite[Pag. 300]{Rob}.

\begin{claim}\label{c.source} $\theta_0$ is a source of $g^n_{x_0}$.
\end{claim}

\begin{proof} Since, by construction, when any expansion exists within $E^c$, the greatest expansion is attained at $\theta_0$, we have that
$$ \Vert D g^n_{x}|_{E^{c}(x,\theta_0)}\Vert  \,\geqslant \,\Vert D g^n_{x}|_{E^{c}(x,y)}\Vert\, , \qquad \forall \, (x,y) \in \mathbb{T}^2\times \mathbb{T}^2 ,\quad \forall \, n \in \mathbb{N}.$$
On the other hand, if $H^{-1}(x_0,y)$ is a non-degenerated interval then $\lambda^c_+(x_0,y) \geqslant 0$ (recall that $\mathcal{E} \subset \mathcal{A}$). As $(x_0,\theta_0)$ is a fixed point of $F_S^n$, the Lyapunov exponent $\lambda^c(F^n_S)(x_0,\theta_0)$ is well defined and satisfies
$$
\lambda^c(F^n_S)(x_0,\theta_0)=
n\, \limsup_{k\,\to\,+\infty}\,\frac1{nk} \,\log \Vert D g^{nk}_{x_0}|_{E^{c}(x_0,\theta_0)}\Vert
=n\,\lambda^c_+(F_S)(x_0,\theta_0)
 \geqslant  n\, \lambda^c_+(F_S)(x_0,y) \geqslant 0.
$$
Thus
$\Vert D g^n_{x_0}|_{E^{c}(x_0,\theta_0)}\Vert\,\geqslant \,1$.
But, as $F_S \in \mathfrak{R}$, one must have
$\| D g^n_{x_0}|_{E^{c}(x_0,\theta_0)}\|> 1$, and so $\theta_0$ is indeed a source of $g^n_{x_0}$.
\end{proof}

\smallskip

\begin{claim}\label{c.newsaddles} The map $g^n_{x_0}$ has three fixed points in $W^s(\theta_0,L^n)$, namely $\theta_0$ and two new saddle points $\theta_1$ and $\theta_2$, one in each connected component of $W^s(\theta_0,L^n) \setminus \{\theta_0\}$.
\end{claim}

\smallskip

\begin{proof}
We recall the ball $B_{\varrho}(q)\subset \mathbb{T}^2$ and the subset $W_2\subset \mathbb{T}^2$ introduced in Subsection~\ref{sse.construction}.
Since $H^{-1}(x_0,y)$ is non-degenerate  interval, there exists $0 \leqslant i \leqslant n$ such that $\Phi^i(x_0)\in B_{\varrho}(q)$. By construction, outside the set $\{\Phi^i(x_0)\}\times W_2$ introduced in Subsection~\ref{sse.construction}, the slope of the graph of the restriction of the map
$$g^n_{\Phi^i(x_0)} \colon \,\Big\{\Phi^i(x_0)\Big\} \times \mathbb{T}^2 \quad \to \quad \Big\{\Phi^i(x_0)\Big\} \times \mathbb{T}^2$$
to $W^s(\theta_0, L^n)$ is smaller than one. Therefore, there must exist two fixed points by the dynamics $g^n_{\Phi^i(x_0)}$, say $\theta^i_1$ and $\theta^i_2$, on each side of $\theta_0$ inside $W^s(\theta_0, L^n)$. The points $\theta_1$ and $\theta_2$ we were looking for are obtained intersecting the orbits of $\theta^i_1$ and $\theta^i_2$ with the fibre $\{x_0\}\times \mathbb{T}^2$.
\end{proof}

Note that both $(x_0,\theta^i_1)$ and $(x_0,\theta^i_2)$ are hyperbolic periodic points of $F_S$. Furthermore, the fixed points $\theta^i_1$ and $\theta^i_2$ of $g^n_{\Phi^i(x_0)}$ in $\Big\{\Phi^i(x_0)\Big\} \times\mathbb{T}^2$ are the unique saddles inside this set fixed by $g^n_{\Phi^i(x_0)}$.
Indeed, denoting by $[\theta_0, w_2]\subset \Big\{\Phi^i(x_0)\Big\} \times\mathbb{T}^2$ the closure of the connected component of $\big(W^s(\theta_0,L^n)\setminus\{\theta_0\}\big)\cap W_2$ containing the saddle $\theta^i_1$ (the corresponding notation for $\theta^i_2$ is $[-w_2, \theta_0]$) and identifying all the fibers $\Big\{\Phi^j(x_0)\Big\}\times\mathbb{T}^2$ with $\mathbb{T}^2$, we deduce that each one-dimensional maps
$$f_{\Phi^i(x_0)}\colon \, [\theta_0, w_2] \quad \to \quad [\theta_0, w_2]$$
for $i=0,1,\cdots,n-1$ is a preserving orientation concave function (including, possibly, affine components, as happens when $\Phi^j(x_0) \notin B_{\varrho}(q)$) such that
\begin{itemize}
\item $f_{\Phi^i(x_0)}(\theta_0)=\theta_0$;
\smallskip
\item  $f_{x_0}(w_2)=f_{\Phi^j(x_0)}(w_2)$, for every $j \in \{0,\cdots,n-1\}$;
\smallskip
\item there is $i \in \{0,1,\cdots,n-1\}$ such that the restriction $f_{\Phi^i(x_0)}|_{(\theta_0, w_2)}$ has a unique (saddle) fixed point (different from $\theta_0$). 
\end{itemize}
Similarly, for every $i=0,1,\cdots,n-1$, the map
$$f_{\Phi^i(x_0)}\colon \, [-w_2,\theta_0]\quad \to \quad [-w_2,\theta_0]$$
preserves orientation and is concave, which ensures the existence of a unique saddle $\theta^i_2$ inside $(-w_2,\theta_0)$ which is fixed by $g^n_{\Phi^i(x_0)}$. Consequently, apart from $\theta_0$, the points $\theta^i_1$ and $\theta^i_2$ are the unique fixed points of $g^n_{\Phi^i(x_0)}$ in $\Big\{\Phi^i(x_0)\Big\} \times\mathbb{T}^2$.

\smallskip

\begin{claim}\label{claim:attractor}
The non-wandering set of  $g^n_{x_0}$ is given by $\Omega(g^n_{x_0}) = \{\theta_0\}\cup \Lambda^n_{x_0}$, where $\Lambda^n_{x_0}$ is a hyperbolic attractor of topological dimension one.
\end{claim}

\smallskip

\noindent \emph{Proof.}
Note that, regarding the splitting $E^u(L)\oplus E^s(L)$ of the tangent space $T\,\mathbb{T}^2$, the derivative of each $f_{\Phi^i(x_0)}$ is determined by a matrix $Df_{\Phi^i(x_0)}=(a_{ij}) $, which is lower triangular since $a_{11}=\lambda_u$  and $a_{12} = 0$ for the whole family $(f_x)_{x \, \in \, \mathbb{T}^2}$. Thus,
\begin{equation}\label{eq:derivative}
D g^n_{x_0}(y) = \begin{pmatrix}
(\lambda_u)^n &  0\\
b_{21}(y) &  b_{22}(y)
 \end{pmatrix}
\end{equation}
with $0 < b_{22} < 1$ at the saddle fixed point $\theta_1$ and $\theta_2$.
Moreover, we can assume $b_{22}(\theta_1)$,  $b_{22}(\theta_2)$ $\le \lambda^n_s$.
Let $ V \subset \mathbb{T}^2$ be a neighborhood of $\theta_0$ not containing $\theta_1$ and $\theta_2$, and such that
\begin{enumerate}
\item $b_{22} > 1$ for $w \in V \quad$ (that is, $g^n_{x_0}$ is an expansion along $E^c$ in $V$);
\smallskip
\item $0 < b_{22} < 1$ for $w \notin g^n_{x_0}(V) \quad $ (that is, $g^n_{x_0}$ is a contraction along $E^c$ outside $g^n_{x_0}(V)$);
\smallskip
\item $g^n_{x_0}(V) \supset V$.
\end{enumerate}
We observe that such a neighborhood $V$ exists (cf. Exercise 7.36 of \cite{Rob})
and $V\subset W^u(\theta_0,g^n_{x_0})$. So it is a local unstable manifold of $\theta_0$ and $W^u(\theta_0,g^n_{x_0}) = \bigcup_{i \, \geqslant \, 1}\, g^{i\,n}_{x_0}(V).$
Let $N = \mathbb{T}^2\setminus V$. Then $N$ is a trapping region because $g^n_{x_0}(V) \supset V$. Set
$$\Lambda^n_{x_0} \,\eqdef\, \bigcap_{i\, \geqslant \, 1}\,g^{i\,n}_{x_0}(N).$$
This is an attracting set and $\Lambda^n_{x_0} = \mathbb{T}^2 \setminus W^u(\theta_0,g^n_{x_0}).$ Thus, $\Omega(g^n_{x_0}) = \{\theta_0\} \cup \Lambda^n_{x_0}$.

\smallskip

We are left to show that $\Lambda^n_{x_0}$ is hyperbolic. Due to \eqref{eq:derivative}, $E^s(L) = E^c(F_S)$ is an invariant bundle and every vector in this bundle is contracted by $D_z\,g^n_{x_0}$ for $z \in \Lambda^n_{x_0}$. This is precisely the stable bundle on $\Lambda^n_{x_0}$. Let $C>0$ be a global upper bound of $|b_{21}|$. Consider $\alpha = C\,[(\lambda_u)^n - (\lambda_s)^n]^{-1}$ and take the cones
$$\mathcal{C}\,\eqdef\, \big\{(v_1,v_2) \in E^u(L)\oplus E^s(L): \,\,\vert v_2\vert  < \alpha\,\vert v_1\vert\big \}·$$
Then it can be checked, using the lower triangular nature of the derivative of $f_x$, that these cones are invariant and
$$E^u(g^n_{x_0}, z) = \bigcap^{\infty}_{i=1}\, D_{g^{-j n}_{x_0}(z)}\, g^{j n}_{x_0}\,\,\Big(\mathcal{C}\,\big(g^{-j n}_{x_0}(z)\big)\Big)$$
is an invariant bundle on which the derivative is an expansion for every point $z \in \Lambda^n_{x_0}$. This provides the unstable bundle on $\Lambda^n_{x_0}$, hence completing the hyperbolic splitting at the points of this set. This ends the proofs of the last claim, of Lemma~\ref{l.DA} and of Proposition~\ref{prop.DA}.
\end{proof}

\begin{cor}\label{cor:periodic} For every $n \in \mathbb{N}$ and every $(x,y) \in \mathrm{Per}_n(\Phi \times L)$, we have
$$1 \,\leqslant \,\# \big(H^{-1}(x,y) \cap\,\, \mathrm{Per}_n(F_S)\big) \,\leqslant \,3.$$
In particular,
$$\# \mathrm{Per}_n(\Phi\times L) \,\leqslant \# \mathrm{Per}_n(F_S) \,\leqslant \,3 \,\,\#\mathrm{Per}_n(\Phi\times L)  \qquad \forall \,n \in \mathbb{N}$$
thus
$$
\lim_{n\to +\infty}\frac{1}{n}\,\log \#\, \mathrm{Per}_n(F_S)=h_{\mathrm{top}}(F_S).
$$
\end{cor}

\begin{proof}
From Proposition~\ref{prop.DA}, given $x \in \mathrm{Per}_n(\Phi)$, either $g^n_{x}$ is Anosov or a Derived from Anosov.
In the former case, the interval $H^{-1}(x,y)$ is a point. In the latter, the interval $H^{-1}(x,\theta_0)$ associated to the fixed point $(x,\theta_0)$ has exactly three fixed points by $g^n_{x}$.  We also must to estimate the cardinality of $H^{-1}(x,y)\cap \mathrm{Per}_n(F_S)$ when $(x,y)$  is different from of the fixed point $(x,\theta_0)$. The last equality is due to Proposition~\ref{prop.Shub-NY}.

\smallskip

\begin{claim}\label{claim:periodic}
Let $(x,y) \in \mathrm{Per}_n(\Phi\times L)$ and suppose that $g^n_{x}$ is a Derived from Anosov. If $y\neq \theta_0$, then $H^{-1}(x,y)$ is a point.
\end{claim}

\begin{proof}
Suppose, on the contrary, that $H^{-1}(x,y)$ is a non-degenerated interval. Then
$$F_S^n:\, H^{-1}(x,y)\quad \to \quad H^{-1}(x,y)$$
is a Morse-Smale diffeomorphism of this interval (recall that $F_S \in \mathfrak{R}$). Since $F_S$ is a preserving orientation map, the boundary points of the interval $H^{-1}(x,y)$, say $(x,a_1)$ and $(x,a_2)$, are necessarily fixed by $F_S^n$. This implies, using the fact that $H^{-1}(x,\theta_0)\cap  H^{-1}(x,y)=\emptyset$, that
$$\Big\{(x,a_1), \, (x,a_2)\Big\} \subset \{x\}\times\Omega(g^n_{x}) \setminus \{(x,\theta_0)\} = \{x\}\times \Lambda^n_{x}$$
and therefore $(x,a_1)$ and $(x,a_2)$ are two sinks of $F_S^n|_{H^{-1}(x,y)}$. This forces the existence of a third point
$$(x,a_3) \,\,\in \,\,H^{-1}(x,y) \setminus \Big\{(x,a_1),\,(x,a_2)\Big\}$$
such that $F_S^n(x,a_3)=(x,a_3)$ and $(x,a_3)$ is a source of $F_S^n|_{H^{-1}(x,y)}$. But $(x,a_3)$ also belongs to $\{x\}\times\Omega(g^n_{x}) \setminus \{(x,\theta_0)\} = \{x\}\times \Lambda^n_{x}$, so this conclusion contradicts Claim~\ref{claim:attractor}.
\end{proof}

Finally, we observe that, for every $n \in \mathbb{N}$,
\begin{equation}\label{e.union}
\mathrm{Per}_n(F_S) \,=\, H^{-1} \big(\mathrm{Per}_n(\Phi \times L)\big) \cap \mathrm{Per}_n(F_S) \,=\, \bigcup_{(x,y)\,\in \,\mathrm{Per}_n(\Phi \times L)} \,\,H^{-1}(x,y) \cap \mathrm{Per}_n(F_S).
\end{equation}
Thus, $\,\, \#\, \mathrm{Per}_n(F_S)\,\leqslant \#\, \mathrm{Per}_n(F_S) \leqslant \,3 \,\,\#\mathrm{Per}_n(\Phi\times L)$ for every $n \in \mathbb{N}$, as claimed.
\end{proof}

\smallskip
The proof of Theorem~\ref{teo.A} is complete.


\smallskip

\begin{rem}\label{re.Kal}

Every $F_S$ belonging to the residual $\mathfrak{R}$
satisfies the conditions:
\begin{enumerate}
\item[(1)] All the periodic points of $f$ are hyperbolic (in particular, the set $\mathrm{Per}_n(f)$ is finite for every positive integer $n$).
\medskip
\item[(2)] There exists $K > 0$ such that $\,\,\frac{1}{n}\,\log \#\, \mathrm{Per}_n(f)  \leqslant \log K\,\,$ for every $n \in \mathbb{N}$.
\end{enumerate}
So, the class of skew product we consider provides a local residual subset where both properties hold. We note that according to~\cite{Ka1999} the set of $C^r$ diffeomorphisms for which the properties (1) and (2) are valid is dense in the space of $C^r$ diffeomorphisms, $r \geqslant 1$.
\end{rem}

\subsection{Proof of Theorem~\ref{teo.A} (b)}\label{prop.weak-limit}

We now prove that the measure $\mu_{\mathrm{max}}$ is the weak$^*$ limit of the sequence of probability measures on $\mathbb{T}^2\times \mathbb{T}^2$
\begin{equation*}\label{eq:distribution}
n \in \mathbb{N} \quad \mapsto \quad \mu_n\,\,\eqdef\,\, \frac1{\#\,\mathrm{Per}_n(F_S)} \, \sum_{(x,y)\,\in\,\mathrm{Per}_n(F_S)}\,\delta_{(x,y)}.
\end{equation*}
To do it, we will show that $\mu_{\mathrm{max}}$ is the unique weak$^*$ limit point of the sequence $(\mu_n)_n$.

Consider the sequence of probabilities $(\nu_n)_{n \, \in \, \mathbb{N}}$ on $\mathbb{T}^2\times \mathbb{T}^2$ defined by
$$n \in \mathbb{N} \quad \mapsto \quad  \nu_n\,\,\eqdef\,\, \frac1{\#\,\mathrm{Per}_n(\Phi \times L)} \,
 \sum_{(x,y)\,\in\,\mathrm{Per}_n(\Phi \times L)}\,\delta_{(x,y)}. $$
We know that this sequence of measures converges in the weak$^*$ topology to the  measure of maximal entropy $\nu_{\mathrm{max}}$ of $\Phi\times L$.


\begin{prop}\label{l.calor} The sequence $\big(\,H_*(\mu_n)\,\big)_{n \, \in \, \mathbb{N}}$ converges to $\nu_{\mathrm{max}}$ in the weak$^*$ topology.
\end{prop}

To prove this assertion it is sufficient to show that the weak$^*$ limit of any convergent sub-sequence of $\big(\,H_*(\nu_n)\,\big)_n$ is equal to $\nu_{\mathrm{max}}$. This will be a consequence of the following two statements.

\begin{lem}\label{l.1}
Let $f: X\to  X$ be a continuous map defined on a compact metric space $(X,d)$. Consider two sequences of $f$-invariant Borel probability measures $(\eta_k)_{k \, \in \, \mathbb{N}}$ and $(\zeta_k)_{k \, \in \, \mathbb{N}}$ on $X$ satisfying
\begin{equation}\label{e.noite}
\exists \,C>1: \quad \quad C^{-1}\,\zeta_k\,\,\leqslant \,\,\eta_k\, \, \leqslant \,\, C\,\zeta_k \qquad \forall \, k \in \mathbb{N}.
\end{equation}
Assume that $(\zeta_k)_{k \, \in \, \mathbb{N}}$ and $(\eta_k)_{k \, \in \, \mathbb{N}}$ converge in the weak$^*$ topology to probability measures $\zeta$ and $\eta$ respectively. Then
$C^{-1}\,\zeta\,\,\leqslant \,\,\eta\,\,  \leqslant\, \, C\,\zeta$. In particular, $\zeta$ and $\eta$ are equivalent.
\end{lem}

\begin{lem}\label{l.2}
If $\eta$ and $\zeta$ are $f$-invariant probability measures on $X$ such that $\eta$ is ergodic and $\zeta$ is absolutely continuous with respect to $\eta$, then $\zeta=\eta$.
\end{lem}

Let us postpone for the moment the proofs of these lemmas to complete the proof of Proposition~\ref{l.calor}.

\begin{proof}[Proof of Proposition~\ref{l.calor}]

Using the fact that for every $(x,y)\in \mathbb{T}^2\times \mathbb{T}^2$ we have $H_*\delta_{(x,y)} = \delta_{H(x,y)}$, we deduce from \eqref{e.union} that the $( \Phi\times L)$-invariant probability measure $H_*(\mu_n)$ satisfies
\begin{equation*}
\begin{split}
H_*(\mu_n)&=\frac1{\#\,\mathrm{Per}_n(F_S)} \sum_{(x,y)\,\in\,\mathrm{Per}_n(\Phi\times L)}\,\#\,\big( H^{-1}(x,y)\cap \mathrm{Per}_n(F_S)\big)\,\,\delta_{(x,y)} \\
& = \left(\frac{\#\,\mathrm{Per}_n(\Phi\times L)}{\#\,\mathrm{Per}_n(F_S)}\right)\,\, \frac1{\#\,\mathrm{Per}_n(\Phi\times L)} \,
 \sum_{(x,y)\,\in\,\mathrm{Per}_n(\Phi\times L)} \,\#\,\big( H^{-1}(x,y)\cap \mathrm{Per}_n(F_S)\big)\,\,\delta_{(x,y)}.
\end{split}
\end{equation*}
Besides, after Corollary~\ref{cor:periodic} we know that
$$ \forall\,  n \in \mathbb{N}, \quad \quad 1\leqslant\, \#\,\big( H^{-1}(x,y)\cap \mathrm{Per}_n(F_S)\big)\,\leqslant 3 \quad \quad \mbox{and} \quad \quad
\frac1{3} \leqslant  \frac{\#\,\mathrm{Per}_n(\Phi\times L)}{\#\,\mathrm{Per}_n(F_S)} \leqslant 1.$$
Thus,
\begin{equation*}\label{e.ineq}
\forall\,  n \in \mathbb{N}, \quad \forall \,\text{ Borel set $A\subset \mathbb{T}^2\times \mathbb{T}^2$} \quad \Rightarrow  \quad \frac1{3}\,\nu_n(A) \leqslant H_*(\mu_n) (A) \leqslant  3\,\nu_n(A).
\end{equation*}
Let $\eta_k:=H_*(\mu_{n_k})$ be a subsequence converging to a probability measure $\nu_0$ in the weak$^*$ topology. Since $\zeta_k:=\nu_{n_k}$ converges to $\nu_{\mathrm{max}}$, it follows from Lemma~\ref{l.1} that $\nu_0$ and $\nu_{\mathrm{max}}$ are equivalent measures. On the other hand, as $\nu_{\mathrm{max}}$ is ergodic, Lemma~\ref{l.2} implies that $\nu_0=\nu_{\mathrm{max}}$.
\end{proof}

We now return to the proof of the two pending lemmas.

\begin{proof}[Proof of Lemma~\ref{l.1}]
By symmetry of the inequality~\eqref{e.noite}  it is enough  to check that for every open set $U$ of $ \mathbb{T}^2\times \mathbb{T}^2$ we have $\eta(U)\,\leqslant C \,\zeta(U).$ Indeed, due the regularity of the measures $\zeta$ and $\eta$, from the previous inequality we get, for every Borel set $A$ in $\mathbb{T}^2\times \mathbb{T}^2$,
\begin{equation*}
\eta(A) = \inf\,\{\eta(G):  \,G\,\, \mbox{is open and}\,\, A \subset G\} \,\,  \leqslant \,\, C\,\inf\,\{\zeta(G): \, G\,\, \mbox{is open and}\,\, A \subset G\} = C\,\zeta(A).
\end{equation*}
So, $\zeta(A)=0$ implies $\eta(A)=0$.

Now, consider the sequence of closed sets in $ \mathbb{T}^2\times \mathbb{T}^2$ defined by
$$k \, \in\, \N \quad \mapsto \quad F_k = \Big\{x \in X : \,d(x, X\setminus U) \geqslant \frac{1}{k}\Big\}.$$
From Uryshon's Lemma there exists a continuous function $g_k : X \to [0,1]$ such that
\begin{equation*}\label{e.ury}
\mathbbm{1} _{F_k} \leqslant g_k \leqslant \mathbbm{1}_U\,,  \qquad \forall \, k \in \N.
\end{equation*}
We may assume that $g_k$ converges to $\mathbbm{1} _U$ in a monotonic and increasing way. Thus,
\begin{eqnarray*}
\eta(U) &=& \sup_k \int g_k \, d\eta \qquad \quad \text{(by the Monotone Convergence Theorem)}\\
&=& \sup_k \,\lim_n \int g_k \, d\eta_n \qquad \quad \text{(by the weak$^*$ convergence of $(\eta_n)_{n \,\in \, \mathbb{N}}$)}\\
&\leqslant& C\,\sup_k \,\lim_n \int g_k \, d\zeta_n \qquad \quad \text{(by equation~\eqref{e.noite})}\\
&=& C\,\sup_k \int g_k \, d\zeta \qquad \quad \text{(by the weak$^*$ convergence of $(\zeta_n)_{n \,\in \, \mathbb{N}}$)}\\
&=& C\,\zeta(U) \qquad \quad \text{(by the Monotone Convergence Theorem)}.
\end{eqnarray*}
\end{proof}

\begin{proof}[Proof of Lemma~\ref{l.2}]
Consider a Borel set $A\subset X$. By Birkhoff's Ergodic Theorem we have
$$\phi_A(x)\,:=\, \lim_{n\,\to\,+\infty}\,\frac1{n}\,\Big\{0 \leqslant j \leqslant n-1: \,f^{j}(x) \in A \Big\}= \mu(A)$$
for $\mu$-almost every $x \in X$, and $\nu(A)=\int \phi_A(x)\,d\nu(x)$. Since $\nu \ll \mu$, we also get  $\phi_A(x)=\mu(A)$ for $\nu$-almost every $x$. So,
$\int \phi_A(x)\,d\nu(x)=\mu(A)$. Hence $\nu(A)=\mu(A)$.
\end{proof}

\begin{cor}\label{cor.cl}
The sequence $(\mu_n)_{n \,\in \, \mathbb{N}}$ converges to $\mu_{\mathrm{max}}$ in the weak$^*$ topology.
\end{cor}

\begin{proof}
We will show that $\mu_{\mathrm{max}}$ is the unique weak$^*$ accumulation point of the sequence $(\mu_n)_{n \in \mathbb{N}}$. Suppose that the subsequence $(\mu_{n_k})_{k}$ converges to a probability measure $\mu_0$. We will verify that $h_{\mu_0}(F_S)=h_{\mathrm{top}}(F_S)$, and so, by the uniqueness of the measure of maximal entropy of $F_S$, we deduce that $\mu_0=\mu_{\mathrm{max}}$.

Using Proposition~\ref{prop.Shub-NY} (b) and Ledrappier-Walters' formula, if follows that
\begin{equation}\label{e.canicula}
h_{\eta}(F_S)=h_{H_*(\eta)}(\Phi\times L), \quad \quad \forall \, \eta \,\in \,\mathscr{P}(\mathbb{T}^2 \times \mathbb{T}^2, \, F_S).
\end{equation}
Now, from Proposition~\ref{l.calor} and the continuity of $\eta\to H_*(\eta)$, it follows that
$H_*(\mu)=\nu_{\mathrm{max}}$. Then, by \eqref{e.canicula} and Proposition~\ref{prop.Shub-NY} (a), we obtain
$$h_{\mu}(F_S)=h_{H_*(\mu)}(\Phi\times L)=h_{\nu_{\mathrm{max}}}(\Phi\times L)=h_{\mathrm{top}}(\Phi\times L)=h_{\mathrm{top}}(F_S).$$
\end{proof}

\section{Proof of Theorem~\ref{teo.B}}\label{se.proof-Theorem-B}

We start observing that, as the periodic points of $F_S$ are hyperbolic, the set of periodic points of $F_S$ is countable, and so zero dimensional. Besides, $F_S$ has the small boundary property (cf. \cite[Subsection 2.1]{B2017}; it was proved in \cite{L1995} that on a finite dimensional manifold any dynamical system whose set of periodic points is countable have this property). Moreover, as already mentioned, the central direction of $F_S$ is one-dimensional, and so $F_S$ is entropy-expansive. After summoning Remark~\ref{re.Kal} and Proposition~\ref{prop.Shub-NY} (a), to show the existence of a principal strongly faithful symbolic extension with embedding for $F_S$ we are left to control of the growth rate of the periodic points with the period at arbitrarily small scales.

\begin{lem}\label{le:per-exp} If $F_S$ belongs to the residual $\mathfrak{R}$, then $F_S$ is asymptotically per-expansive.
\end{lem}

\begin{proof} Given $\varepsilon>0$ and $(x_0,y_0)\in \mathbb{T}^2\times\mathbb{T}^2$, define
$$B^{F_S}_{\infty, \,\varepsilon}(x_0,y_0) := \Big\{\,\,(x,y) \in \mathbb{T}^2\times\mathbb{T}^2\,: \,\,d(F_S^i(x,y),F_S^i(x_0,y_0)) \leqslant \varepsilon, \quad \forall \,i \in \mathbb{Z}\,\,\Big\}.$$
We claim that
$$\forall \, \varepsilon > 0\quad \forall \, n \in \mathbb{N} \quad \forall \, (x_0,y_0) \in \mathbb{T}^2\times\mathbb{T}^2 \quad \quad \#\,\left(\mathrm{Per}_n(F_S) \cap B^{F_S}_{\infty, \,\varepsilon}(x_0,y_0)\right) \leqslant 3.$$
Firstly note that the central foliation of $F_S$ is plaque expansive (cf. \cite[Sections 7 \& 8]{HPS}): there exists $\varepsilon_0 > 0$ such that, for every $0 < \varepsilon < \varepsilon_0$, if $(x,y)$ belongs to $B^{F_S}_{\infty, \,\varepsilon}(x_0,y_0)$, then both points $(x_0,y_0)$ and $(x,y)$ lie on the same leaf of the central foliation, which is sent by the semi-conjugation $H$ into a stable leaf. On the other hand, if $\mathrm{Per}_n(F_S) \cap B^{F_S}_{\infty, \,\varepsilon}(x_0,y_0) \neq \emptyset$ then $x_0$ is periodic and so, by Proposition~\ref{prop.DA}, $g^n_{x_0}$ is an Anosov or a Derived from Anosov. In the former
 case, $B^{F_S}_{\infty, \,\varepsilon}(x_0,y_0)\subset B^{F^n_S}_{\infty, \,\varepsilon}(x_0,y_0)=\{(x_0,y_0)\}$. In the latter case, the intersection cannot have more than three periodic points: otherwise, if we assume the existence of at least four elements in $\mathrm{Per}_n(F_S)$ in $B^{F_S}_{\infty, \,\varepsilon}(x_0,y_0)$, then we may find two hyperbolic point $(x_0,y_1)$ and $(x_0,y_2)$ in $\mathrm{Per}_n(F_S) \cap B^{F_S}_{\infty, \,\varepsilon}(x_0,y_0)$ such that $H(x_0,y_1)\neq H(x_0,y_2)$ are in $\mathrm{Per}_n(\Phi\times L)$ and belong to the same stable leaf of $\Phi \times L$. This contradicts the known dynamics within stable leaves.
\end{proof}

To end the proof of Theorem~\ref{teo.B} we just make a straightforward application of the Main Theorem of \cite{B2017}.

\section{Proof of Theorem~\ref{teo.C}}\label{se.proof-Theorem-C}

Suppose that $\Phi$ is a linear hyperbolic automorphism of $\mathbb{T}^2$ and let $\nu_{\mathrm{SRB}}$ be the SRB measure of $\Phi\times L$. Denote by $\mu_{\mathrm{max}}$ and $\nu_{\mathrm{max}}$ the measures of maximal entropy of $F_S$ and $\Phi \times L$, respectively. Similarly, let $\mu_{\mathrm{SRB}}$ and $\nu_{\mathrm{SRB}}$ the SRB measures of $F_S$ and $\Phi \times L$.

Consider the expanding eigenvalues $\beta_1 > 1$ and $\beta_2 > 1$ of $\Phi$ and $L$, respectively. By Pesin's formula, the topological entropy of $\Phi\times L$ is given by
$$h_{\mathrm{top}}(\Phi\times L)=\log \,\beta_1 + \log\,\beta_2.$$
Note also that, on the corresponding regular sets, the positive Lyapunov exponents $\lambda^{uu} > \lambda^{u}>0$ of $\mu_{\mathrm{max}}$ and $\nu_{\mathrm{max}}$ are given by
$$\lambda^{uu}\big(\mu_{\mathrm{max}}\big) = \lambda^{uu}\big(\nu_{\mathrm{max}}\big) = \log\,\beta_1 \,\,\quad \text{and} \,\,\quad \lambda^{u}\big(\mu_{\mathrm{max}}\big) = \lambda^{u}\big(\nu_{\mathrm{max}}\big) = \log\,\beta_2.$$
Moreover, under the assumption that both $\Phi$ and $L$ are linear automorphisms of $\mathbb{T}^2$, the measure $\nu_{\mathrm{max}}$ coincides with the Lebesgue measure in $\mathbb{T}^2\times\mathbb{T}^2$ (cf. \cite[Theorem 8.15]{W1981}). Besides, $\nu_{\mathrm{max}}=\nu_{\mathrm{SRB}}$ since the mapping $(x,y)\,\mapsto\, J^u_{\Phi\times L}(x,y)$ is constant and equal to $\beta_1\,\beta_2$, and so
$$h_{\nu_{\mathrm{SRB}}}(\Phi\times L)= \int\,\log\, J^u_{\Phi\times L}\,\,d\nu_{\mathrm{SRB}} = \log \,\beta_1 + \log\, \beta_2 = h_{\mathrm{top}}(\Phi\times L).$$

To prove Theorem~\ref{teo.C} (a), we note that, by construction, for every $(x,y) \in \mathbb{T}^2 \times \mathbb{T}^2$ we have
$$J^u_{F_S}(x,y)\geqslant \beta_1\,\beta_2.$$
So, $J^u_{F_S}$ and $J^u_{\Phi\times L}$ satisfy the assumption \eqref{e.eq1} of Proposition~\ref{prop.SRB}. Therefore, one has
$$H_*(\mu_{\mathrm{SRB}}) = \nu_{\mathrm{SRB}}.$$
To prove Theorem~\ref{teo.C} (b), we use Proposition~\ref{prop.Shub-NY} (a) to deduce that
$$h_{\mu_{\mathrm{SRB}}}(\Phi\times L)= \int\,\log\, J^u_{F_S}\,\,d\mu_{\mathrm{SRB}} \geqslant \log \,\beta_1 + \log\,\beta_2 = h_{\mathrm{top}}(\Phi\times L)=h_{\mathrm{top}}(F_S)$$
and thereby conclude that $h_{\mu_{\mathrm{SRB}}}(\Phi \times L) = h_{\mathrm{top}}(F_S)$, as claimed. The fact that $\mu_{\mathrm{SRB}}$ is also the unique physical measure of $F_S$ has already been proved in Proposition~\ref{prop.minimal}.






\begin{thebibliography}{999}\label{ref}

\bibitem{AS1970}
R.~Abraham and S.~Smale.
\newblock \emph{Non-genericity of $\Omega$-stability.}
\newblock Global Analysis, Volume XIV of Proc. Symp. in Pure Math. (Berkeley 1968), Amer. Math. Soc., Providence, R.I., 1970, 5--8.

\bibitem{An2010}
M.~Andersson.
\newblock \emph{Robust ergodic properties in partially hyperbolic dynamics.}
\newblock Trans. Amer. Math. Soc. 362 (2010) 1831--1867.

\bibitem{BV2000}
C.~Bonatti and M.~Viana.
\newblock \emph{SRB measures for partially hyperbolic systems whose central direction is mostly contracting.}
\newblock Israel J. Math. 115 (2000) 157--193.

\bibitem{BDF}
C.~Bonatti, L.~D\'iaz and T.~Fisher.
\newblock \emph{Super-exponential growth of the number of periodic orbits inside homoclinic classes.}
\newblock Discrete Contin. Dyn. Syst. 20:3 (2008) 589--604.

\bibitem{BDV2005}
C.~Bonatti, L.~D\'iaz and M.~Viana.
\newblock \emph{Dynamics Beyond Uniform Hyperbolicity.}
\newblock Encyclopaedia of Mathematical Sciences 102, Springer-Verlag Berlin Heidelberg, 2005.

\bibitem{BFF2002}
M.~Boyle, D.~Fiebig and U.~Fiebig.
\newblock \emph{Residual entropy, conditional entropy and subshift covers.}
\newblock Forum Math. 14 (2002) 713--757.


\bibitem{B1970}
R.~Bowen.
\newblock \emph{Topological entropy and Axiom A.}
\newblock  Global Analysis, Volume XIV of Proc. Symp. in Pure Math. (Berkeley 1968), Amer. Math. Soc., Providence, R.I., 1970, 23--41.

\bibitem{B1971}
R.~Bowen.
\newblock \emph{Periodic points and measures for Axiom A diffeomorphisms.}
\newblock Trans. Amer. Math. Soc. 154 (1971) 377--397.


\bibitem{B1975}
R.~Bowen.
\newblock \emph{Equilibrium states and the ergodic theory of Anosov diffeomorphisms.}
\newblock Lecture Notes in Mathematics 470, Springer-Verlag Berlin Heidelberg, 1975.

\bibitem{B2011}
D.~Burguet.
\newblock \emph{$C^2$ surface diffeomorphism have symbolic extensions.}
\newblock Invent. Math.  186:1 (2011) 191--236.

\bibitem{BF2013}
D.~Burguet and T.~Fisher.
\newblock \emph{Symbolic extensions for partially hyperbolic dynamical systems with $2$-dimensional center bundle.}
\newblock Discrete Contin. Dyn. Syst. 33:6  (2013) 2253--2270.

\bibitem{B2017}
D.~Burguet.
\newblock \emph{Embedding asymptotically expansive systems.}
\newblock Monatsh. Math. 184:1 (2017) 21--49.

\bibitem{Burguet2017}
D.~Burguet.
\newblock \emph{Periodic expansiveness of smooth surface diffeomorphisms and applications.}
\newblock arxiv:1705.08832v1 (2017)


\bibitem{Bz2005}
J.~Buzzi.
\newblock \emph{Subshifts of quasi-finite type.}
\newblock Invent. Math. 159:2 (2005) 369--406.

\bibitem{BFSV2012}
J.~Buzzi, T.~Fisher, M.~Sambarino and C.~V\'asquez.
\newblock \emph{Maximal entropy measures for certain partially hyperbolic, derived from Anosov systems.}
\newblock Ergod. Th. \& Dynam. Sys. 32:1 (2012) 63--79.

\bibitem{CP2018}
M.~Carvalho, S.A.~P\'erez.
\newblock \emph{Equilibrium states for a class of skew products}.
\newblock Ergod. Th. \& Dynam. Sys. (2019) http://dx.doi.org/10.1017/etds.2019.32

\bibitem{CY2005}
W.~Cowieson and L.-S.~Young.
\newblock \emph{SRB measures as zero-noise limits.}
\newblock Ergod. Th. \& Dynam. Sys. 25:4 (2005) 1115--1138.

\bibitem{DF2011}
L.~D\'iaz and T.~Fisher.
\newblock \emph{Symbolic extensions and partially hyperbolic diffeomorphisms.}
\newblock Discrete Contin. Dyn. Syst. 29 (2011) 1419--1441.

\bibitem{DFPV2012}
L.~D\'iaz, T.~Fisher, M.J.~Pac\'ifico and J.~Vieitez.
\newblock \emph{Entropy-expansiveness for partially hyperbolic diffeomorphisms.}
\newblock Discrete Contin. Dyn. Syst. 32:12 (2012) 4195--4207.

\bibitem{DN2005}
T.~Downarowicz and S. Newhouse.
\newblock \emph{Symbolic extensions and smooth dynamical systems.}
\newblock Invent. Math. 160:3 (2005) 453--499.


\bibitem{Fr}
J.~Franks
\newblock \emph{Anosov diffeomorphisms.}
\newblock Global Analysis, Volume XIV of Proc. Symp. in Pure Math. (Berkeley 1968), Amer. Math. Soc., Providence, R.I., 1970, 61--93.

\bibitem{G2016}
K.~Gelfert.
\newblock \emph{Horseshoes for diffeomorphisms preserving hyperbolic measures.}
\newblock Math. Z. 282 (2016) 685--701.

\bibitem{Guc1976}
J.~Guckenheimer.
\newblock \emph{A strange, strange attractor.}
\newblock  In: \emph{The Hopf bifurcation and its applications}, J. E. Marsden and M. McCracken (eds.), Applied Mathematical Sciences 19 (1976) 368--391.


\bibitem{HamPot}
A.~Hammerlindl and R.~Potrie.
\newblock \emph{Classification of systems with center-stable tori.}
\newblock Michigan Math. J. 68 (2019) 147--166.

\bibitem{HHTU2012}
F.R.~Hertz, M.A.~Hertz, A.~Tahzibi and R.~Ures.
\newblock \emph{Maximizing measures for partially hyperbolic systems with compact center leaves.}
\newblock Ergod. Th. \& Dynam. Sys. 32:2 (2012) 825--839.

\bibitem{HHU2012}
F.R.~Hertz, M.A.~Hertz and R.~Ures.
\newblock \emph{A non-dynamically coherent example on $T^3$.}
\newblock Ann. Inst. H. Poincar\'e Anal. Non Lin\'eaire 33 (2012) 1023--1032.

\bibitem{HPS}
M.~Hirsch, C.~Pugh and M.~Shub.
\newblock \emph{Invariant Manifolds.}
\newblock Lecture Notes in Mathematics 583, Springer-Verlag Berlin Heidelberg, 1977.


\bibitem{K1980}
A.~Katok.
\newblock \emph{Lyapunov exponents, entropy and periodic orbits for diffeomorphisms.}
\newblock  Publ. Math. I.H.E.S. 51 (1980) 137--174.

\bibitem{Ka1999}
V.Yu.~Kaloshin.
\newblock \emph{An extension of the Artin-Mazur Theorem.}
\newblock Annals of Math. 150:2 (1999) 729--741.

\bibitem{LW1977}
F.~Ledrappier and P.~Walters.
\newblock \emph{A relativised variational principle for continuous transformations.}
\newblock  J. London Math. Soc. 16:3 (1977) 568--576.

\bibitem{LY1985}
F.~Ledrappier and L.-S.~Young.
\newblock \emph{The metric entropy of diffeomorphisms. I. Characterization of measures satisfying Pesin's entropy formula.}
\newblock Annals of Math. 122:3 (1985) 509--539.

\bibitem{L1995}
E.~Lindenstrauss.
\newblock \emph{Lowering topological entropy.}
\newblock J. d'Analyse Math. 67 (1995) 231--267.


\bibitem{Man}
A.~Manning.
\newblock \emph{ On zeta functions and Anosov diffeomorphisms.}
\newblock Thesis, University of Warwick, 1972.


\bibitem{Mz1976}
M.~Misiurewicz.
\newblock \emph{Topological conditional entropy.}
\newblock Studia Math. 55:2 (1976) 175--200.




\bibitem{NY1983}
S.~Newhouse and L.-S.~Young.
\newblock \emph{Dynamics of certain skew products.}
\newblock Lecture Notes in Mathematics 1007, Springer-Verlag Berlin Heidelberg, 1983, 611--629.

\bibitem{O1968}
V.~Oseledets.
\newblock \emph{A multiplicative ergodic theorem: Lyapunov characteristic numbers for dynamical systems.}
\newblock Trans. Moscow Math. Soc. 19 (1968) 197--231.

\bibitem{Pes1977}
J.~Pesin.
\newblock \emph{Characteristic Lyapunov exponents and smooth ergodic theory.}
\newblock Russian Math. Surveys 32:4 (1977) 55--114.

\bibitem{Pot}
R.~Potrie.
\newblock \emph{Partial hyperbolicity and attracting regions in 3-dimensional manifolds.}
\newblock  PhD. Thesis (2012), arXiv:1207.1822.

\bibitem{PujSam2006}
E.~Pujals and M.~Sambarino.
\newblock \emph{A sufficient condition for robustly minimal foliations.}
\newblock Ergod. Th. \& Dynam. Sys. 26:1 (2006) 281--289.


\bibitem{Rob}
C.~Robinson.
\newblock \emph{Dynamical Systems. Stability, Symbolic Dynamics, and Chaos.}
\newblock Second edition. Studies in Advanced Mathematics. CRC Press, Boca Raton, FL, 1999.

\bibitem{Ru1978}
D.~Ruelle.
\newblock \emph{An inequality for the entropy of differentiable maps.}
\newblock Bol. Soc. Bras. Mat. 9 (1978) 83--87.

\bibitem{S1971}
M.~Shub.
\newblock \emph{Topological transitive diffeomorphisms in $T^4$.}
\newblock Lecture Notes in Mathematics 206, Springer-Verlag Berlin Heidelberg, 1971, 39--40.

\bibitem{U2012}
 R.~Ures.
\newblock \emph{Intrinsic ergodicity of partially hyperbolic diffeomorphisms with hyperbolic linear part.}
\newblock Proc. Amer. Math. Soc. 140:6 (2012) 1973--1985.

\bibitem{W1981}
P.~Walters.
\newblock \emph{An Introduction to Ergodic Theory.}
\newblock Graduate Texts in Mathematics 79, Springer-Verlag New York, 1981.

\bibitem{Y1981}
L.-S.~Young.
\newblock \emph{On the prevalence of horseshoes.}
\newblock Trans. Amer. Math. Soc. 263 (1981) 75--88.

\bibitem{Y2002}
L.-S.~Young.
\newblock \emph{What are SRB measures and which dynamical systems have them?}
\newblock J. Stat. Phys. 108 (2002) 733--54.


\end{thebibliography}
\end{document}